\documentclass[11pt,reqno]{amsart}
\usepackage{amssymb,latexsym}
\usepackage{graphicx}
\usepackage{url}		%does nice formatting of URLs

\usepackage{qtree}
\usepackage{mathtools}

\calclayout

\makeatletter
\newcommand{\monthyear}[1]{%
  \def\@monthyear{\uppercase{#1}}}
\def\@monthyear{\uppercase{}}
\newcommand{\volnumber}[1]{%
  \def\@volnumber{\uppercase{#1}}}
   \def\@volnumber{\uppercase{ }}
\AtBeginDocument{%
\def\ps@plain{\ps@empty
  \def\@oddfoot{\@monthyear \hfil \thepage}%
  \def\@evenfoot{\thepage \hfil \@volnumber}}
\def\ps@firstpage{\ps@plain}
\def\ps@headings{\ps@empty
  \def\@evenhead{%
    \setTrue{runhead}%
    \def\thanks{\protect\thanks@warning}%
    \uppercase{}\hfil}%
  \def\@oddhead{%
    \setTrue{runhead}%
    \def\thanks{\protect\thanks@warning}%
    \hfill\uppercase{Rational Enumerations}}%
  \let\@mkboth\markboth
  \def\@evenfoot{%
    \thepage \hfil \@volnumber}%
  \def\@oddfoot{%
    \@monthyear \hfil \thepage}%
  }%
\footskip=25pt
\pagestyle{headings}%
}
\makeatother

\newcommand{\Q}{{\mathbb Q}}

\newcommand{\Z}{{\mathbb Z}}

\theoremstyle{plain}
\numberwithin{equation}{section}
\newtheorem{thm}{Theorem}[section]
\newtheorem{theorem}[thm]{Theorem}

\newtheorem{corollary}[thm]{Corollary}

\begin{document}
%% replace the values in the next three lines by the correct information
%\monthyear{Month Year}
%\volnumber{Volume, Number}
\setcounter{page}{1}

\title{Re$^3$counting the Rationals}
\author{Sam Northshield}
\address{Department of Mathematics\\
                SUNY-Plattsburgh\\
                Plattsburgh, NY 12901}
                \email{northssw@plattsburgh.edu}
%\thanks{Research supported in part by SUNY-Plattsburgh travel grant.   Thanks also to hospitality of organizers of the Fibonacci Conference of 2018.}

\begin{abstract}
In 1999, Neil Calkin and Herbert Wilf wrote ``Recounting the rationals" 
which gave an explicit bijection between the positive integers and the positive rationals.    We find several different (some new) ways to construct this enumeration and thus create pointers for generalizing.   Next, we use circle packings to generalize and find two other enumerations.    Surprisingly, the three enumerations are all that are possible by using this technique. 
 The proofs involve, among other things, ``negative" continued fractions,
Chebyshev polynomials, Euler's totient function, and generalizations of Stern's diatomic sequence.   Finally we look at some of the remarkable similarities -- and differences -- of these sequences.   \end{abstract}

\maketitle

\section{Introduction}

In 1999, Neil Calkin and Herbert Wilf wrote their charming ``Recounting the rationals" \cite{CW}
which gave an explicit listing of the positive rationals beginning thus:

{\small $$\frac11,\frac21,\frac 12, \frac31, \frac23, \frac 32, \frac13, \frac41, \frac 34, \frac53, \frac25, \frac52, \frac35, \frac43, \frac14, \frac51, \frac45, \frac74, \frac37, \frac83, \frac58,\frac75, \frac27, \frac72, \frac57, \frac85, \frac38, \frac73, \frac47, \frac54, \frac15, \frac61,\ldots.$$}
Throughout this paper, we label this sequence $(r_n)_{n=1}^{\infty}$.
 \vskip .1 in
 
There are many  ways to arrive at this sequence.  In Section 2, we present several.   In particular,  we recall the Calkin-Wilf tree which arranges the positive rationals on the vertices of a rooted binary tree.   To show that all positive rationals appear on the tree, we introduce the ``Tree Theorem", a useful tool for enumeration that we also use later in Section 3.       Another natural setting for $(r_n)$ is in terms of Stern's diatomic array and in terms of Stern's diatomic sequence.  We are quickly led to a ``semi-recursive" formula
$$r_n=2\nu_2(n)+1-\dfrac1{r_{n-1}}$$  
where $\nu_2(n)$ is the usual valuation function that counts the number of times 2 divides $n$.

From this, we find a new source for $(r_n)$ in terms of a greedy algorithm:   for a given list of rationals $[l_0,l_1,\ldots,l_k]$, append $l_{k+1}$ defined as the least $N-1/l_k$ not already on the list.    Starting with $[\infty]$, iterates of this algorithm gives $(r_n)$.   

This leads to ``minus continued fractions":  for $$(a_1, a_2, a_3, \ldots ):=a_1-\dfrac1{a_2-\dfrac 1{a_3 -\cdots}},$$
$$r_n=(1),(3,1), (1,3,1), (5,1,3,1),(1,5,1,3,1), \ldots$$
and also 
$$r_n=(1), (2), (1,2), (3), (1,3), (2,2), (1,2,2),\ldots .$$

Finally, since Stern's sequence has closed formulas -- both a Binet formula and a formula as a sum across Pascal's triangle mod 2 \cite{N1} -- we can create closed formulas for $(r_n)$.   

Through all these ways of constructing $(r_n)$ we are led to many possible directions to generalize.    Our initial choice actually goes further afield to varieties of circle packings.    In Section 3, we look at the family of Ford circles and how they are parameterized by the rationals.    It turns out that Stern's sequence appears here.     
More general circle packings are introduced and, from them, we construct analogues of Stern's sequence.    From these, we construct two other sequences of the positive rationals that begin:  
\vskip .1 in
 {\small $$\frac21, \frac11, \frac41, \frac32,\frac23, \frac31, \frac43, \frac12,\frac61,\frac53, \frac45, \frac72, \frac{10}7, \frac35,\frac83, \frac54, \frac25, \frac51, \frac85,\frac34, \frac{10}3, \frac75, \frac47, \frac52,\frac65, \frac13, \frac81,\frac74,\frac67,\frac{11}3,\frac{16}{11},\frac58,\ldots$$}
 \vskip .1 in
and

{\small $$\frac31, \frac21,\frac32,\frac11,\frac61, \frac52, \frac95,\frac43, \frac34,\frac51,\frac{12}5,\frac74,\frac97,\frac23,\frac92,\frac73,\frac{12}7, \frac54,\frac35,\frac41,\frac94,\frac53,\frac65,\frac12,\frac91,\frac83, \frac{15}8, \frac75,\frac67,\frac{11}2,\frac{27}{11},\frac{16}9\ldots.$$}
\vskip .1 in

 Surprisingly, the three enumerations are all that are possible by using the ``circle" technique -- all others necessarily contain irrational numbers as well.
 A fourth is introduced as well;  here we get rationals in $\Q(\phi)^+$ (we conjecture that we get all of them).   

In Section 4, we look at some of the remarkable similarities -- and differences -- of these sequences.   In particular,  semi-recursive formulas, tree inducing functions, greedy algorithms, and generating functions are all studied.    Further, we look at the ``degree" of each sequence (the critical exponent $\delta$ for which $\limsup x_n/n^{\delta}$ is finite).

Ponton \cite{P} has recently found these three sequences as well as some corresponding trees that generate them.   His paper and this one are still quite complementary.

%%%%%%%%%%%

\section{The basic enumeration, and several ways to get it}
In this section, we consider the enumeration of the positive rationals
$$r_n=\frac11,\frac21,\frac 12, \frac31, \frac23, \frac 32, \frac13, \frac41, \frac 34, \frac53, \frac25, \frac52, \frac35, \frac43, \frac14, \frac51, \frac45, \frac74, \frac37, \frac83, \frac58,\frac75, \frac27, \frac72, \frac57, \frac85, \frac38, \frac73, \frac47, \frac54,,\ldots.$$
The reader is invited to find a pattern for generating this sequence before reading further.   

Our first way of constructing this sequence is by arranging the positive rationals on a tree and then reading off the entries.   

\subsection{Tree Theorem}

A function from a countable set $S$ to itself can be thought of as inducing a digraph, with vertex set $S$, where one assigns to  each vertex a directed edge from that vertex to its image under $f$.   It is generally possible to get loops (corresponding to fixed points of $f$) and cycles.   
The following theorem gives a condition that guarantees that the digraph is actually a union of disjoint rooted trees. 

\begin{theorem}[Tree Theorem]
If $S$ is a countable set,  $F:S\rightarrow S$ with set of fixed points $S_0$, and $\Phi: S\rightarrow {\mathbb Z}^+$ such that, for all $x\not\in S_0$,
$$\Phi(F(x))<\Phi(x),$$
then $F$ canonically arranges the elements of $S$
 on the vertices of a collection of rooted directed trees (with set of respective roots $S_0$).
\end{theorem}

\begin{proof}  On the vertex set $S$, 
let ${\mathcal E}$ be the set of all directed edges $[x,F(x)]$ where $x$ ranges over all elements of $S-S_0$ and define $G$ to be the directed graph $(S,{\mathcal E})$.   

Let $F_n$ denote the $n$th iterate of $F$.   For any $x\in S$, the sequence $\Phi(F_n(x))$ is  a decreasing sequence of positive integers but, by the well-ordering of ${\mathbb Z}^+$, cannot be strictly decreasing.   Hence, there exists (a unique) $z\in S_0$ and $N\in{\mathbb Z}^+$ so that $F_N(x)=z$.  Let $\sigma(x)$ denote the minimum such $N$.  
It follows that there is a path of length $\sigma(x)$ from $x$ to $z$.   Since $F$ is a function (so the out-degree of every vertex is at most 1), that path is unique.   Hence, the connected  component of $G$ containing $x$ is a directed tree with root $z$.  The result follows. 
\end{proof}

This theorem is a discrete version of the ``contraction mapping theorem" which states that, no matter the starting point, the iterates of any strict contraction on a set must converge to a fixed point.  

\subsection{Calkin-Wilf Tree}
We introduce a tree that first appeared in  ``Recounting the rationals" \cite{CW} by Calkin and Wilf.  Our approach to showing it contains all the positive rationals will be by the Tree Theorem.  

Consider $F: x\longmapsto \max\left\{\dfrac x{1-x}, x-1\right\}$ and 
$\Phi(a/b)=a+b$.  It is easy to see that for $x\neq 1$, $\Phi(F(x))<\Phi(x)$ and so, by the Tree Theorem, every positive rational appears exactly once on the tree.   

The map $F$ may be inverted in order to get a procedure for actually \emph{constructing} the tree:
$$F^{-1}(x)=\left\{\frac x{1+x}, x+1\right\}.$$
 \begin{figure}[htbp] 
   \centering
   \includegraphics[width=4in]{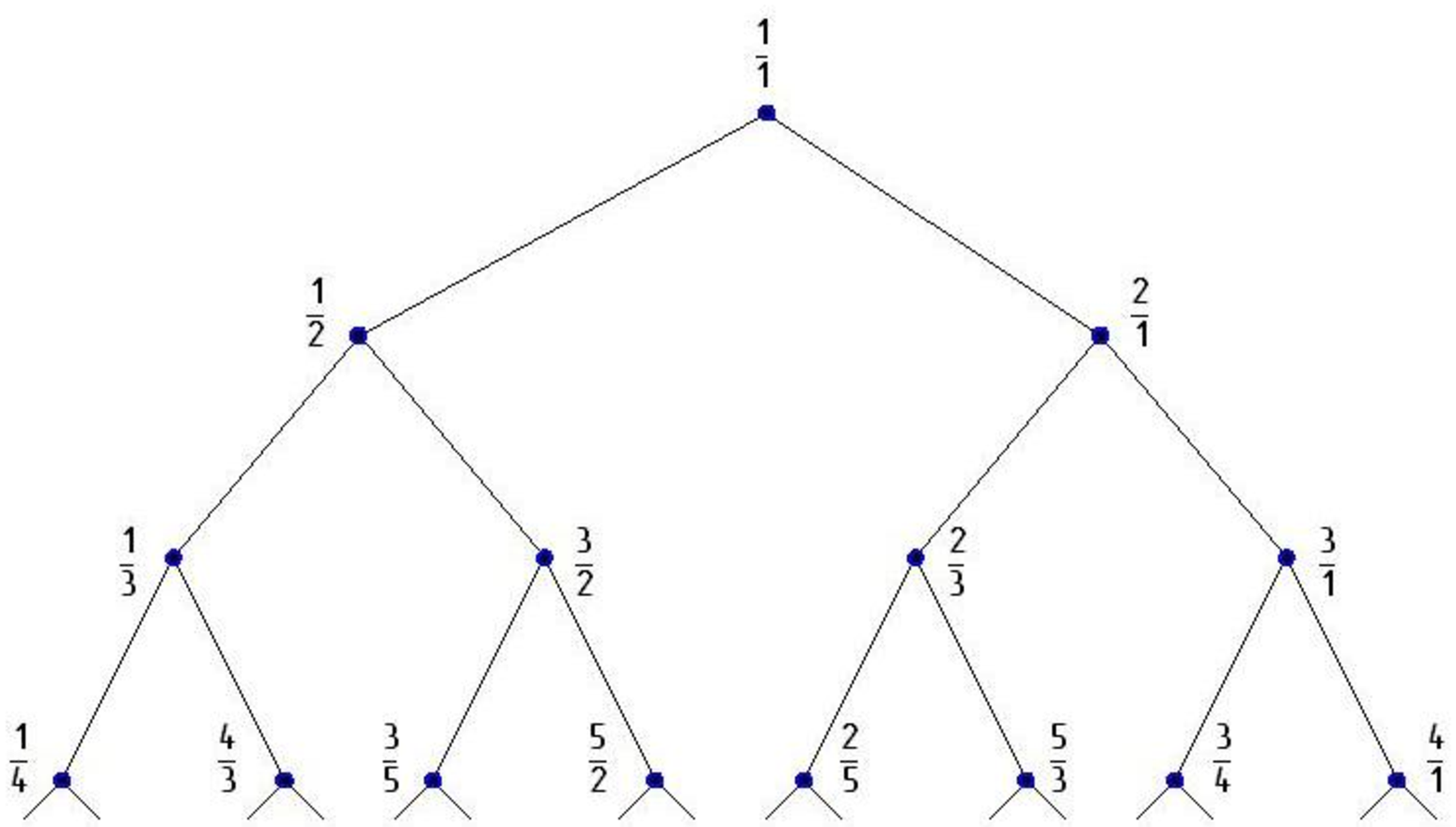} 
 \end{figure}
 
 Reading this tree like a book, but right to left, gives the enumeration of the positive rationals:
   $$\frac11,\frac21,\frac 12, \frac31, \frac23, \frac 32, \frac13, \frac41, \frac 34, \frac53, \frac25, \frac52, \frac35, \frac43, \frac14, \frac51, \frac45, \frac74, \frac37, \frac83, \frac58,\frac75, \frac27, \frac72, \frac57, \frac85, \frac38, \frac73, \frac47, \frac54, \frac15, \frac61,\ldots.$$
\vskip .1 in
\subsection{Stern's diatomic array and sequence}

Probably the clearest setting for understanding the sequence $(r_n)$ is Stern's diatomic array,  sometimes thought of as ``Pascal's triangle with memory", which  begins thus:
$$\begin{array}{cccccccccccccccccc}1&&&&&&&&&&&&&&&&1\\1&&&&&&&&2&&&&&&&&1\\1&&&&3&&&&2&&&&3&&&&1
\\1&&4&&3&&5&&2&&5&&3&&4&&1\\1&5&4&7&3&8&5&7&2&7&5&8&3&7&4&5&1\\.&.&.&.&.&.&.&.&.&.&.&.&.&.&.&.&.\end{array}$$
It is defined recursively as follows:  Start with row 1  1.  Then, given the $n$th row, define the next one by copying
the numbers on the $n$th row but inserting, in each gap, the sum of the two numbers above.   

The numbers in the diatomic array, read like a book (but deleting the right-most column of 1s),  form what is known as Stern's diatomic sequence which begins:
$$a_n=1,1,2,1,3,2,3,1,4,3,5,2,5,3,4,1,5,4,7,3,8,5,7,2,7,5,8,3,7,4,5,1,6,\ldots$$   
It is easy to see that it is defined recursively by
$$a_1=1,\qquad a_{2n}=a_n,\qquad a_{2n+1}=a_n+a_{n+1}\eqno{(1)}$$

Every pair $(a,b)$ of consecutive numbers in the diatomic array corresponds, via $(a,b)\mapsto\frac ab$, to the rationals appearing on the Calkin-Wilf tree
and so $$r_n=\dfrac{a_{n+1}}{a_n}.\eqno{(2)}$$

See \cite{N1} and its references, and \cite{Sl}, sequence A002478, for information about this exceptional, and exceptionally well studied, sequence.

\subsection{A semi-recursive formula}
A ``semi-recursive" formula for $(r_n)$ is a formula of the form $r_n=g(r_{n-1}, d_n)$ where $(d_n)$ is some fixed sequence and $g$ is some function. 

Let $$\nu_k(n):= \max\{j: k^j|n\}$$
be the number of times $k$ divides $n$.  

Since $$a_{2m+2}+a_{2m}=a_{2m+1}$$
and
$$a_{2m+1}+a_{2m-1}=a_{m+1}+a_{m-1}+2a_m,$$
it follows by induction on $\nu_2(n)$ that
$$\dfrac{a_{n+1}+a_{n-1}}{a_n}=2\nu_2(n)+1.$$
From this and equation (2), it follows that 
$$r_n=2\nu_2(n)+1-\dfrac1{r_{n-1}}.\eqno{(3)}$$

\subsection{A recursive formula}

Note that $a_{2n}/a_{2n+1}=a_n/(a_n+a_{n+1})<1$ and so $\lfloor a_{2n}/a_{2n+1}\rfloor=0=\nu_2(2n+1)$.
Further, if $\lfloor a_{n-1}/a_n\rfloor=\nu_2(n)$, then $\lfloor a_{2n-1}/a_{2n}\rfloor=\lfloor (a_{n-1}+a_n)/a_{n}\rfloor=\nu_2(n)+1=\nu_2(2n)$.
By induction,
$$\left\lfloor\dfrac1{r_n}\right\rfloor=\nu_2(n+1).$$
Substituting this into equation (3), we find the recursive formula
$$r_{n+1}=2\left\lfloor\dfrac1{r_n}\right\rfloor+1-\dfrac1{r_n}\eqno{(4)}$$
which can be rewritten as
$$r_{n+1}=1+\dfrac1{r_n}-2\left\{\dfrac1{r_n}\right\}\eqno{(5)}$$
where $\{x\}$ denotes the fractional part of $x$.    This remarkable result  is equivalent to one first noticed by Moshe Newman (see Ponton \cite{P}).

Finally, in terms of Stern's sequence, we may rewrite this as
$$a_{n+1}=a_n+a_{n-1}-2(a_{n-1}\text{ mod } a_n).\eqno{(6)}$$

\subsection{A greedy algorithm}

Define a list of numbers inductively as follows:   given a list $[l_1,l_2,\ldots,\l_k]$,  append 
$l_{k+1}$, defined as the smallest number of the form $n-1/l_k$ ($n\in{\mathbb Z}^+$) not already on the list.   
Starting with the list $[1]$, we get
 $$\frac11,\frac21,\frac 12, \frac31, \frac23, \frac 32, \frac13, \frac41, \frac 34, \frac53, \frac25, \frac52, \frac35, \frac43, \frac14, \frac51, \frac45, \frac74, \frac37, \frac83, \frac58,\frac75, \frac27, \frac72, \frac57, \frac85, \frac38, \frac73, \frac47, \frac54, \frac15, \frac61,\ldots.$$

To see that this sequence is the same as $(r_n)$,   note that by equations (1) and (2), for all $n$,
$$r_{2n}=1+r_n \qquad \text{ and } \qquad r_{2n+1}=1-\frac1{r_{2n}}.$$
Hence,  for every odd $n$,  $r_n<1$ and $r_{2^kn}=k+r_n$.   

Let $m$ be the smallest number for which $l_m\neq r_m$.  
By equation (4), $r_m=M-\frac1{r_{m-1}}$ for some integer $M$ and, for some  $N\le M$, 
$$l_m=N-\frac1{l_{m-1}}=N-\frac1{r_{m-1}}.$$   Hence  $$l_{m/2^j}=r_{m/2^j}=N-j-\frac1{r_{m-1}}$$ is already on the list and thus $M=N$.   

By induction, the greedy algorithm gives the sequence $(r_n)$.

\subsection{Minus continued fractions}
Keeping with the nomenclature of Katok \cite{Katok}, we call expressions of the following form ``minus continued fractions":
$$(a_1,a_2,a_3,a_4,\ldots):=a_1-\dfrac1{a_2-\dfrac1{a_3-\dfrac1{a_4-\cdots}}}.$$

By equation (3), the positive rationals {$$\frac11,\frac21,\frac 12, \frac31, \frac23, \frac 32, \frac13, \frac41, \frac 34,\ldots$$} can be written in
terms of minus continued fractions:
$$(1), (3,1), (1,3,1), (5,1,3,1), (1,5,1,3,1), (3,1,5,1,3,1), (1,3,1,5,1,3,1),  \ldots$$
with the sequence $2\nu_2(n)+1$ in plain sight.  

Minus continued fractions are \emph{not} unique and we have an alternative formulation.
Given a formal list $\ell=[w_1,w_2,\ldots,w_n]$, define $1+\ell:=[1+w_1, w_2,\ldots, w_n]$ and $1*\ell:=[1,w_1,w_2,\ldots,w_n]$.  
Starting with the empty list $\ell_0:=[ \hskip .05 in]$, define
$$\begin{cases}
&\ell_{2n}:=1+\ell_n\\&\ell_{2n+1}:=1*\ell_{2n}\end{cases}\eqno{(8)}$$
We get the sequence of lists
{ $$[1],[2],[1,2],[3],[1,3],[2,2],[1,2,2], [4],[1,4],[2,3],[1,2,3], [3,2], \ldots$$}
When turned into minus continued fractions:
{ $$(1), (2), (1,2), (3), (1,3), (2,2), (1,2,2), (4), (1,4), (2,3), (1,2,3), (3,2), \ldots$$}
we get our enumeration again (exercise for the reader).

\subsection{Closed formulas}
It is easy to see that, by equation (1) (see \cite{N1}, section 3),
the generating function of $(a_{n+1})$ satisfies
$$A(x)=(1+x+x^2)A(x^2)=\prod_{n\ge 1}(1+x^{2^n}+x^{2^{n+1}}).$$
Note: This implies that $a_{n+1}$ counts the number of ``hyperbinary representations" of $n$ (i.e., the number of ways to represent $n$ as a sum of powers of two with no power appearing more than twice).   

\vskip .1 in
This leads to a closed formula:  letting  $\langle a,b\rangle_2$ be 0 or 1 according to whether $a$,$b$ share non-zero digits in their respective binary expansions or not,
$$a_{n+1}=\sum_{a+2b=n}\langle a,b\rangle_2.\eqno{(10)}$$
By a well-known theorem of Kummer ( \cite{Gr}, \cite{Ku} ) stating that $\nu_p\left(\binom{a+b}{b}\right)$ is the number of carries when adding $a$ and $b$ base $p$,
$$\langle a,b\rangle_2=\left[\binom{a+b}{b}  \mod 2\right]$$
and we have the analogous formulas:  the first is the representation of Fibonacci numbers as diagonal sums in Pascal's triangle, the second, due to Carlitz \cite{Car}, 
$$F_{n+1}=\sum_{a+2b=n}\binom{a+b}{b},  \hskip .2 in a_{n+1}=\sum_{a+2b=n}\left[\binom{a+b}{b}  \mod 2\right].$$

 Here is a further similarity with Fibonacci numbers:  Recalling Binet's formula 
$$F_{n+1}=\frac{\phi^{n+1}-\overline\phi^{n+1}}{\phi-\overline\phi}=\sum_{k=0}^n \phi^{k}\overline\phi^{n-k},$$
we have
$$a_{n+1}=\sum_{k=0}^n \sigma^{s_2(k)}\overline\sigma^{s_2(n-k)}\eqno{(11)}$$
where $\sigma$ is a primitive sixth root of unity and $s_2(n)$ is the number of ones in the binary expansion of $n$ (see \cite{N1}, Prop. 4.4).

Therefore we have two closed formulas for $r_n$:
$$r_n=\dfrac{\sum_{a+2b=n}\left[\binom{a+b}{b}  \mod 2\right]}{\sum_{a+2b=n-1}\left[\binom{a+b}{b}  \mod 2\right]}$$
and
$$r_n=\dfrac{\sum_{k=0}^n \sigma^{s_2(k)}\overline\sigma^{s_2(n-k)}}{\sum_{k=0}^{n-1}\sigma^{s_2(k)}\overline\sigma^{s_2(n-1-k)}}.$$

\vskip .2 in
\subsection{Ways to view $(r_n)$}
In conclusion, we summarize some of the places that the enumeration
$$\frac11,\frac21,\frac 12, \frac31, \frac23, \frac 32, \frac13, \frac41, \frac 34, \frac53, \frac25, \frac52, \frac35, \frac43, \frac14, \frac51, \frac45, \frac74, \frac37, \frac83, \frac58,\frac75, \frac27, \frac72, \frac57, \frac85, \frac38, \frac73, \frac47, \frac54, \frac15, \frac61,\ldots.$$
occurs:

\begin{itemize}
\item a tree defined by a contraction on ${\mathbb Q}^+$,\vskip .1 in

\item  $\frac{a_{n+1}}{a_n}$ where $a_n$ is Stern's diatomic sequence,\vskip .1 in

\item a semi-recursive formula $r_n=2\nu_2(n)+1-\frac1{r_{n-1}}$,\vskip .1 in

\item a recursive formula $r_{n+1}=1+\dfrac1{r_n}-2\left\{\dfrac1{r_n}\right\},$\vskip .1 in

\item a greedy algorithm, \vskip .1 in

\item minus continued fractions (two different ways), and \vskip .1 in

\item closed formulas based on those for $(a_n)$.\vskip .1 in

\end{itemize}

%%%%%

\vskip .3 in

\section{Circle packings and some new enumerations}

For real  $x,y$, let $C_{x,y}$ be the circle of center $(x/y, 1/2y^2)$ and radius $1/2y^2$.   
Every such circle is thus above and tangent to the horizontal axis.    By the Pythagorean theorem, 
$C_{x,y}$ and $C_{u,v}$ are tangent to each other -- we write $C_{x,y}||C_{u,v}$ -- if and only if $|xv-yu|=1$.

A M\"obius transformation is a function of the form
$$\begin{pmatrix} a&b\\c&d\end{pmatrix}(z):=\frac{az+b}{cz+d}$$
where the defining matrix is non-singular.   
M\"obius transformations preserve circles and tangencies.   Further,  it is not hard to see (\cite{N2}, Lemma 6) that when $ad-bd=1$,
$$\begin{pmatrix} a&b\\c&d\end{pmatrix}\left(C_{x,y}\right) = C_{ax+by,cx+dy}.$$

\subsection{Ford circles}  

Consider the array of circles
$$C_{0,1}||C_{1,1}||C_{1,0}$$
 \begin{figure}[htbp] 
   \centering
   \includegraphics[width=2.3in]{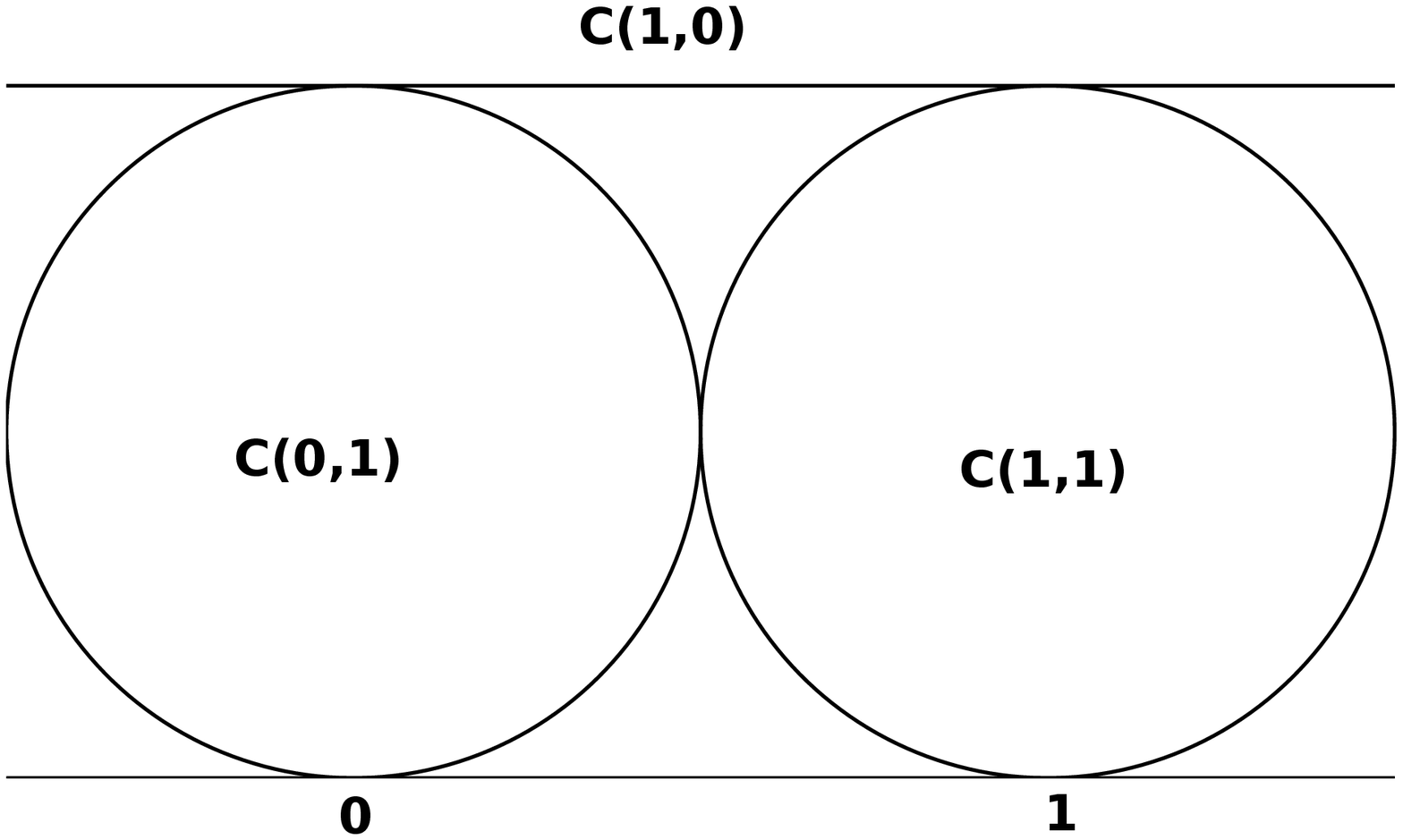}
  \end{figure}

Given circles $C_{0,1}$ and $C_{1,0}$ we formed a new one, $C_{1,1}$, by componentwise addition.
For tangent circles $C_{a,b}||C_{c,d}$,  there is a M\"obius transformation taking $C_{0,1}$ and $C_{1,0}$ to $C_{a,b}$ and $C_{c,d}$, respectively, and thus taking $C_{1,1}$ to $C_{a+c,b+d}$.   
We are thus led to 
the iterative procedure
$$\begin{array}{ccc} C_{a,b}&&C_{c,d}\\
C_{a,b}& C_{a+c,b+d}&C_{c,d}\end{array}$$
Starting with our initial configuration, successive iterations give us
 the ``Ford circles" (see  \cite{F} and \cite{N3}).   

\begin{figure}[htbp] 
\centering{
   \includegraphics[ width=.8 in]{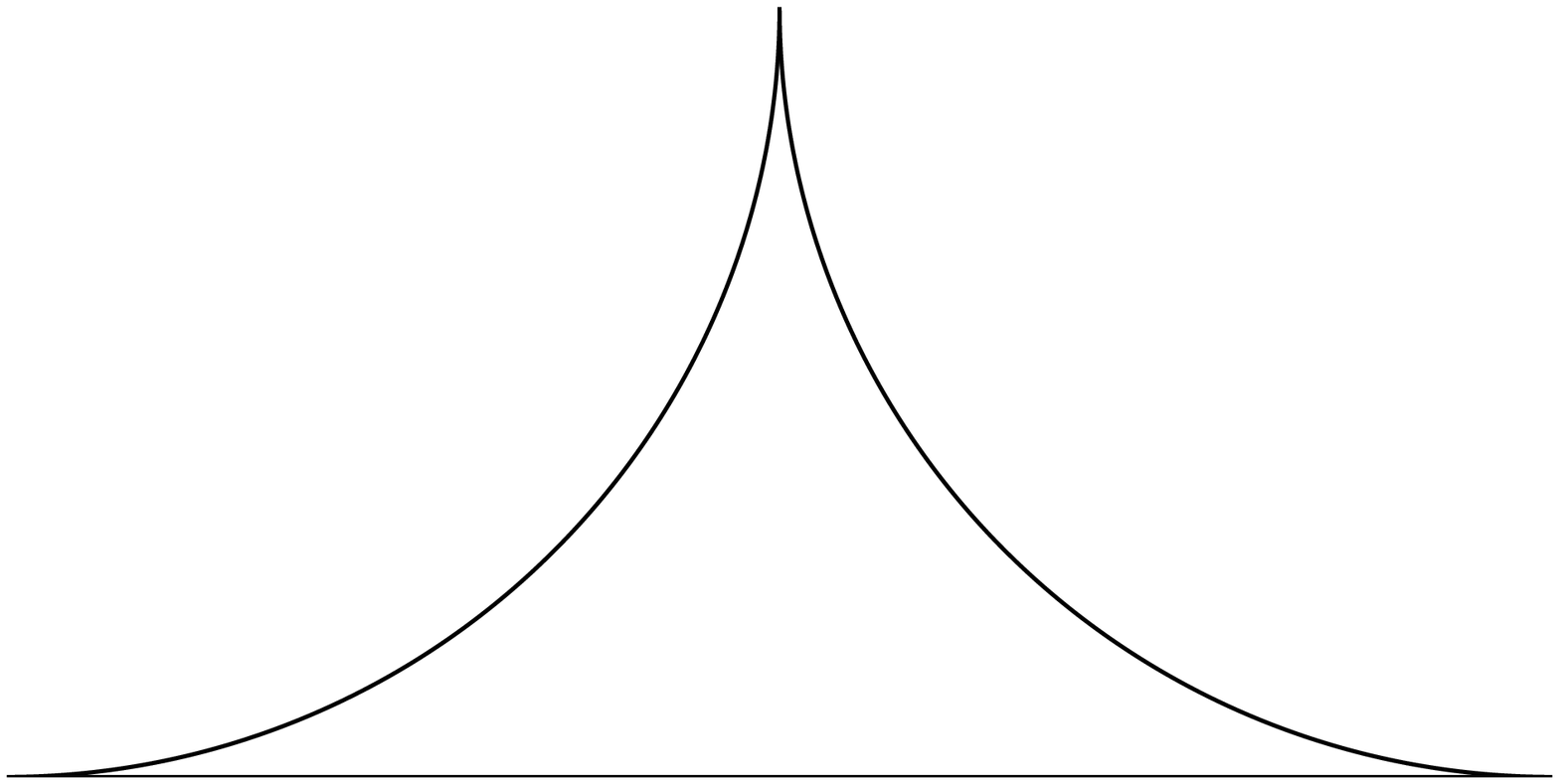} 
   \includegraphics[ width=.8 in]{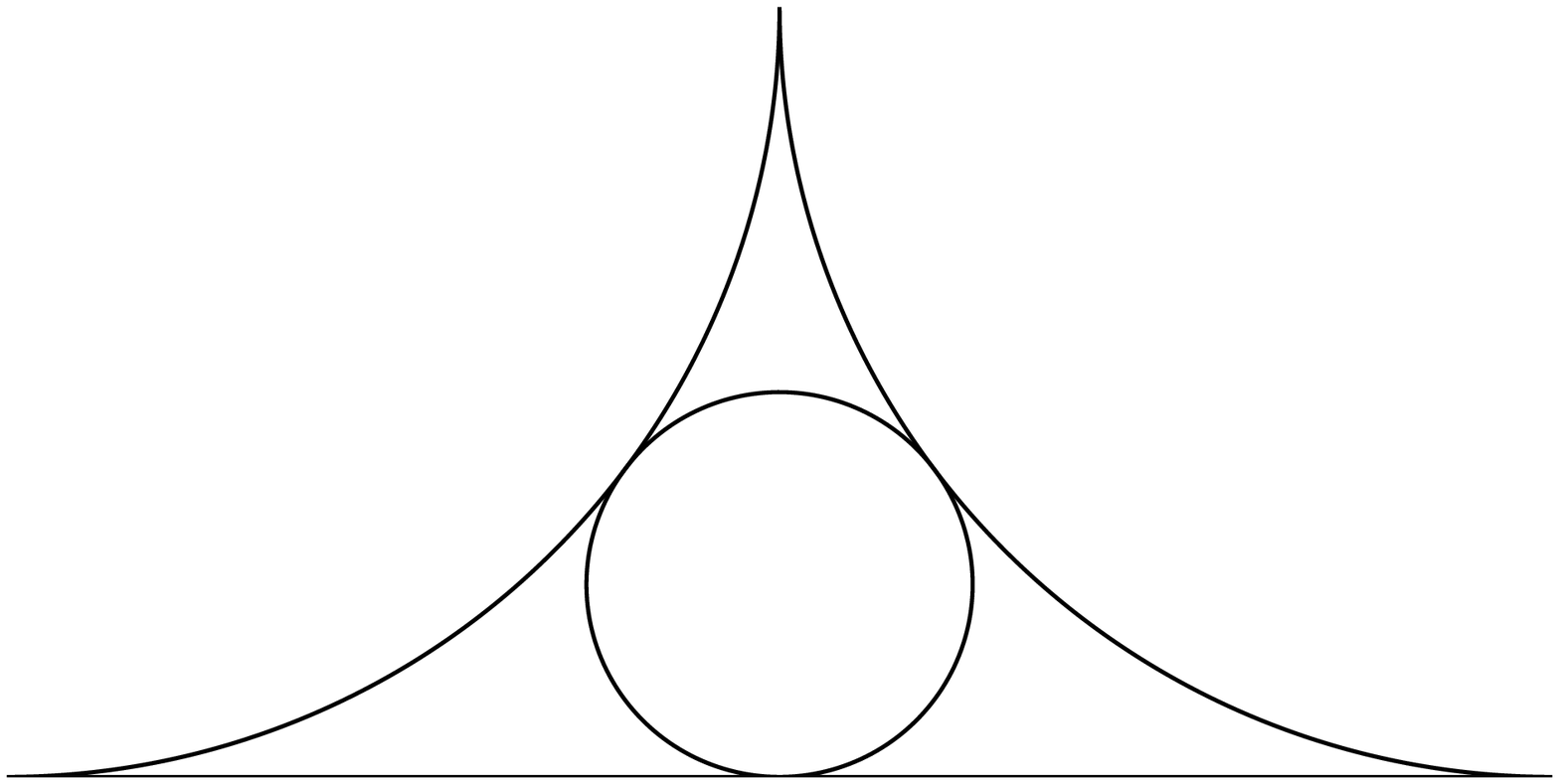} 
   \includegraphics[ width=.8 in]{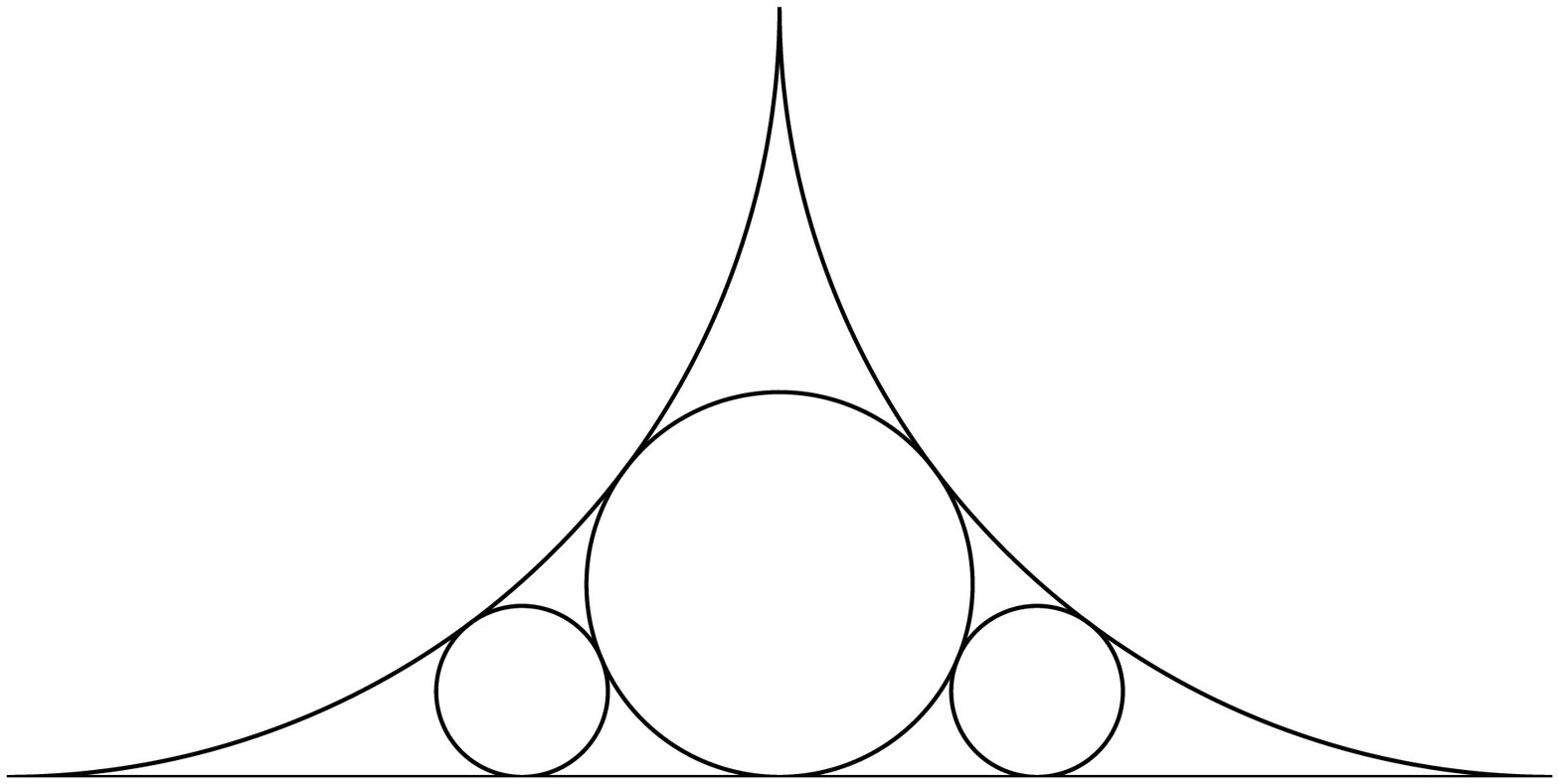}
    \includegraphics[ width=.8 in]{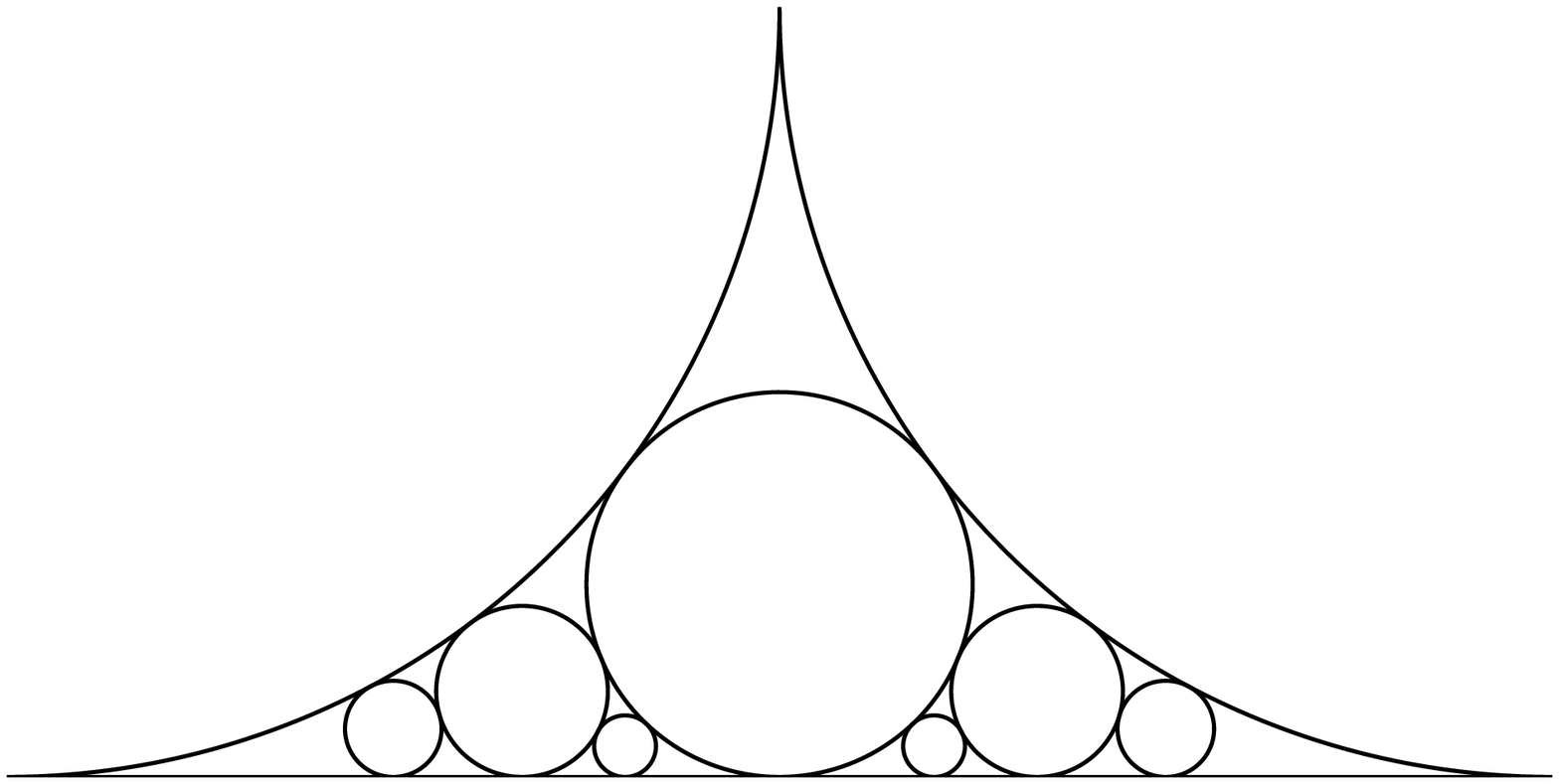} 
    }
        \centering\includegraphics[ width=1.6 in]{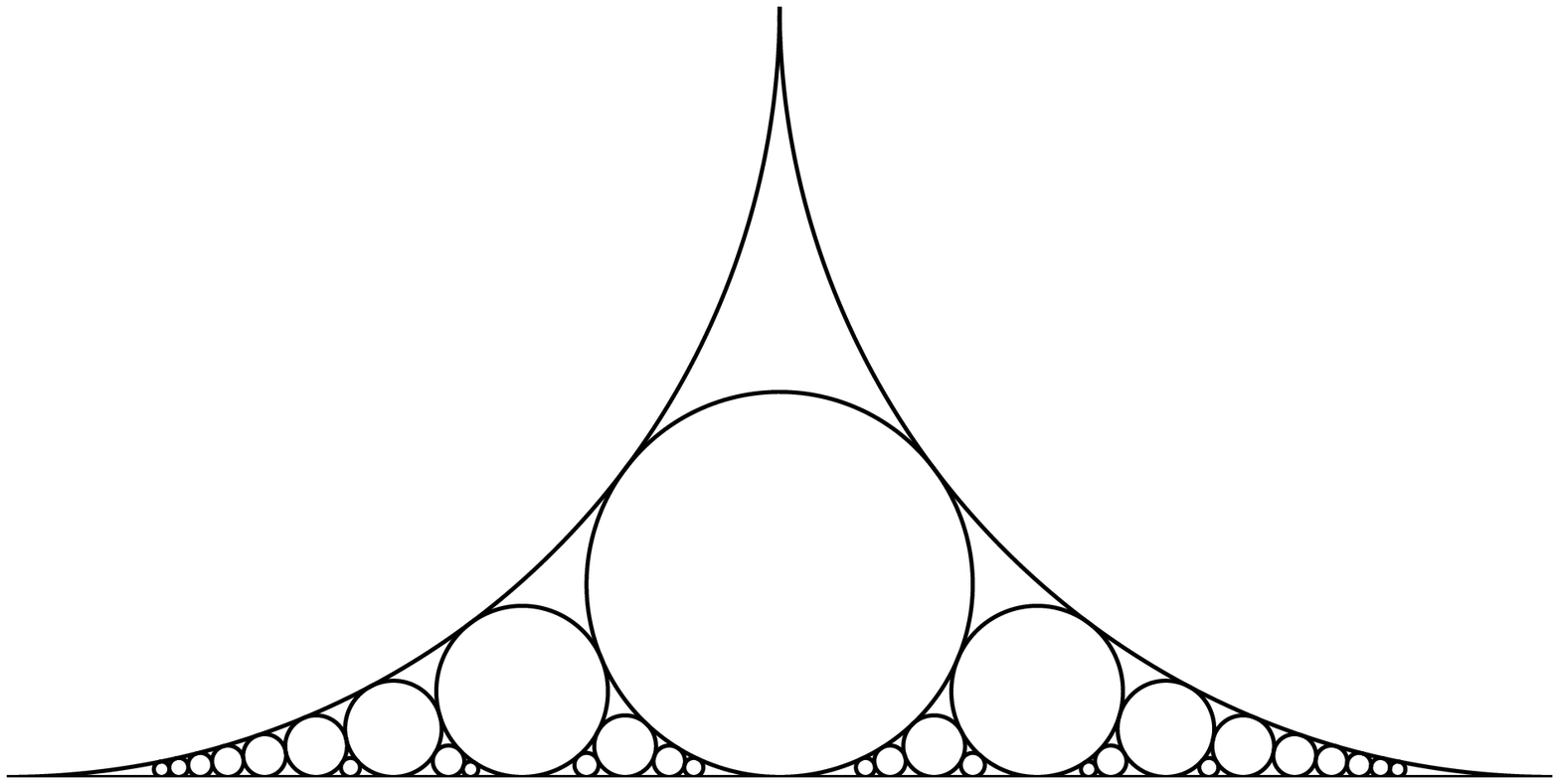} 
\end{figure}

A Ford circle is thus  a circle of the form $C_{a,b}$ where $a,b>0$ are relatively prime integers.   Clearly, any two Ford circles are either tangent to each other or do not touch at all.    Further each is tangent to the horizontal axis at a positive rational number and, in fact,  for every positive rational number $a/b$ in lowest terms there is a Ford circle $C_{a,b}$.

The locations of tangency for these circles are  

\centerline{$\frac01,\frac11$}

\centerline{$\frac01,\frac12, \frac 11$}

\centerline{$\frac01, \frac13, \frac12, \frac23, \frac 11$}

\centerline{$\frac01,\frac14, \frac13,\frac25, \frac12, \frac35, \frac23, \frac34, \frac 11$} \vskip.1 in
\noindent and Stern's sequence makes a new appearance:  	the locations of tangency of the circles in the $k$th iterate are
$a_n/a_{2^k+n},  n=0,\ldots, 2^k.$

Recall, with $r_n:=\frac{a_{n+1}}{a_n}$,  the function $F(x):=\max\left\{\dfrac x{1-x}, x-1\right\}$ is easily seen to satisfy $$F(r_n)=r_{\lfloor n/2\rfloor}$$ and so, by the Tree Theorem, the tree induced by $F$ on the positive rationals coincides with the tree with values $(r_n)$;   i.e., $(r_n)$ enumerates the positive rationals.

\subsection{Another circle packing, diatomic array, and diatomic sequence}
In this subsection, and with subsection 3.1 as a guide,
the array of circles $$C_{0,1}||C_{1,\sqrt 2}||C_{\sqrt 2,1}||C_{1,0}$$

 \begin{figure}[htbp] 
   \centering
   \includegraphics[width=2.3in]{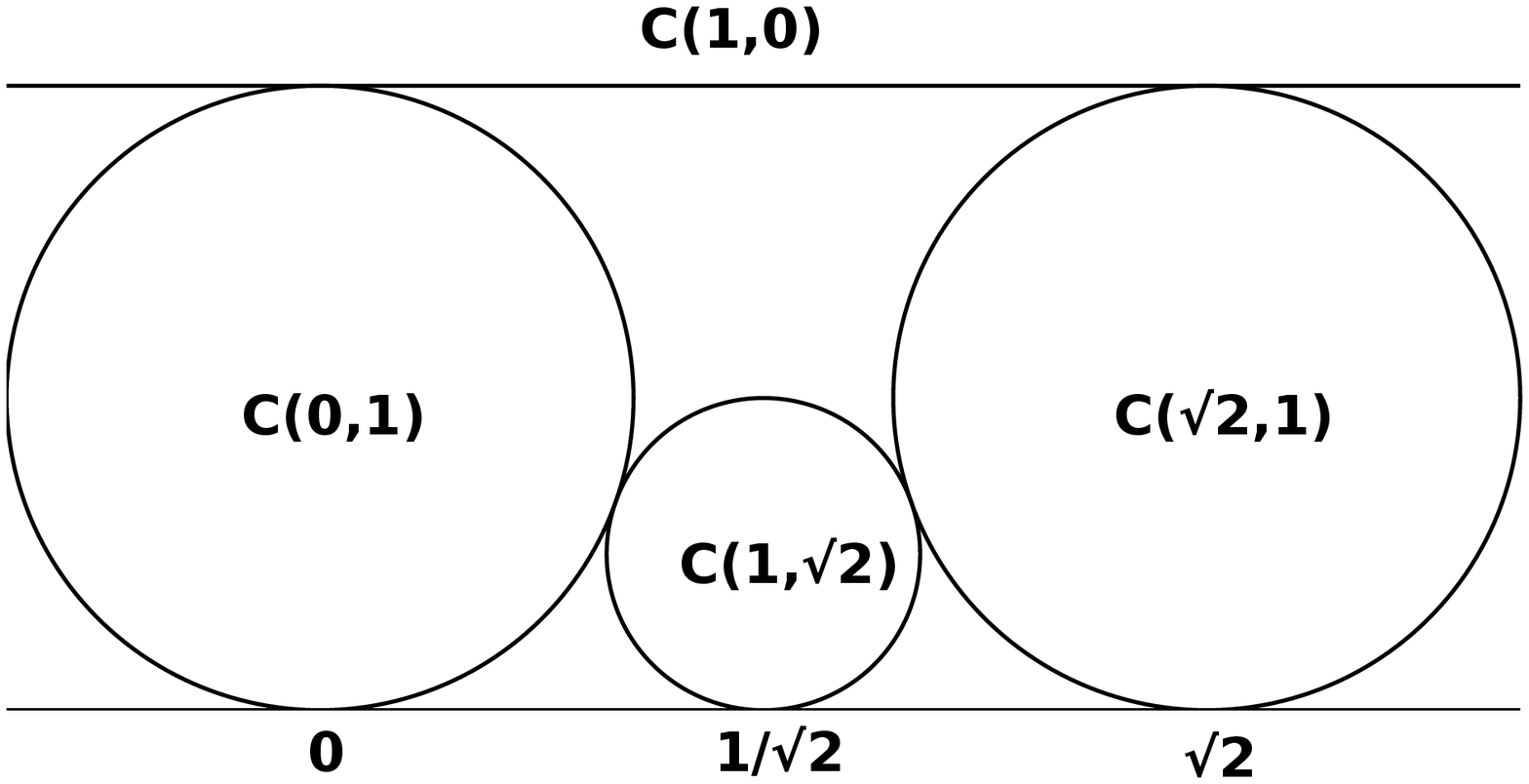}
  \end{figure}
 
   \noindent indicates the recursive process
   
   $$\begin{array}{cccc}
C_{a,b}&&& C_{c,d}\\
C_{a,b}&C_{a\sqrt 2+c, b\sqrt 2+d}&C_{a+c\sqrt 2, b+d\sqrt 2}&C_{b,d}\end{array}$$
\noindent which leads to a function 
$$F(x)=\max\left\{\frac x{1-x}, \frac{2x-2}{2-x}, x-2\right\}$$
that, by the Tree Theorem,  induces an enumeration of the postive rationals.  

Here are the details.  Guettler and Mallows 
proposed a new type of Ford circle in \cite{GM}.   Their form of Ford circles is different from ours since, for a given pair of tangent circles, they add two new tangent circles.  To ensure uniqueness, they do that according to the rule that the tangent points of the four circles with each other all lie on some circle.
\begin{figure}[htbp] 
   \centering
   \includegraphics[width=3in]{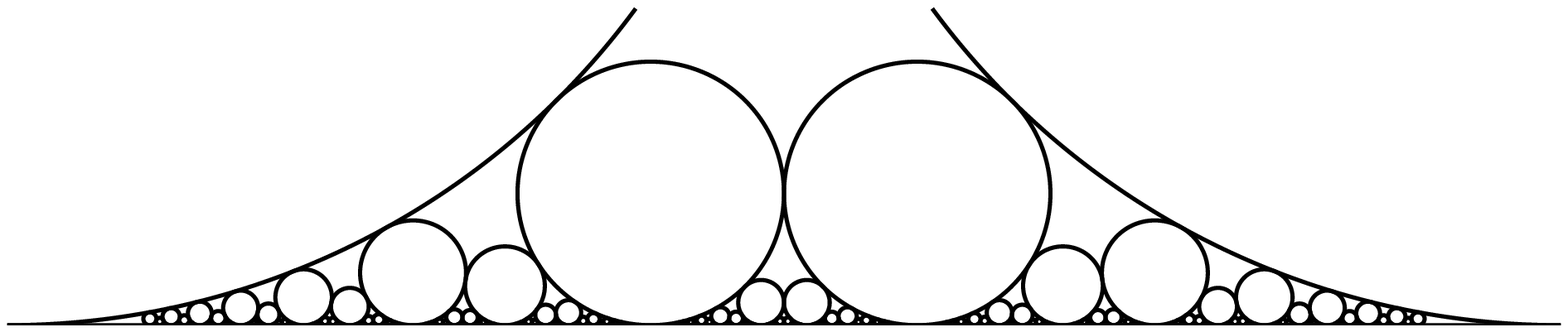} 
 \end{figure}

One can also construct this analogue of the set of Ford circles (with tangents between 0 and $\sqrt 2$)  by, starting with $C_{0,1/\sqrt 2}$ and $C_{\sqrt 2,1}$, recursively appending to every pair of tangent circles  $C_{a,b}, C_{c,d}$, the two other circles $C_{a+\sqrt 2 c,b+\sqrt 2 d}$ and $C_{\sqrt 2  a+c, \sqrt 2 b+d}$.   
 
 The first and second coordinates of these circles satisfy a recursion defined by 
$$\begin{array}{cccccccc}a&&&b&&&a
\\a&a\sqrt 2+b&a+b\sqrt 2&b&b\sqrt2+a&b+a\sqrt2&a\end{array}$$
and we get another diatomic array:

$$\begin{array}{ccccccccccccccccccc}
1&&&&&&&&&&&&&&&&&&1\\
1&&&&&&&&&\sqrt 2&&&&&&&&&1\\
1&&&2\sqrt 2&&&3&&&\sqrt 2&&&3&&&2\sqrt2&&&1\\
1&3\sqrt 2&5&2\sqrt 2&7&5\sqrt 2&3&4\sqrt 2&5&\sqrt 2&5&4\sqrt2&3&5\sqrt2&7&2\sqrt2&5&3\sqrt2&1\\
.&.&.&.&.&.&.&.&.&.&.&.&.&.&.&.&.&.&.\end{array}$$
and another diatomic sequence (formed, again, by deleting the right-most column of 1s):
$$b_n:  0, 1, \sqrt 2, 1,2\sqrt 2,3, \sqrt 2,3, 2\sqrt 2,1,3\sqrt 2,5,2\sqrt 2,7, 5\sqrt 2,3,\ldots$$
Its recursive definition is
$$\begin{cases} b_{3n}&:=b_n,\\ b_{3n+1}&:=\sqrt 2\cdot b_n+b_{n+1}\\
 b_{3n+2}&:=b_n+\sqrt 2\cdot b_{n+1}.\end{cases}\eqno{(12)}$$ 

As with Stern's sequence, this sequence satisfies a 3-term recurrence (Th. 18 of \cite{N2}), namely
$$b_{n+1}=\sqrt2 b_n+b_{n-1}-2(b_{n-1}\text{ mod }\sqrt2 b_n).\eqno{(13)}$$
Hence, for $s_n:=\sqrt 2 b_{n+1}/b_n$,
$$s_{n+1}=2+\dfrac2{s_n}-4\left\{\dfrac1{s_n}\right\}.\eqno{(14)}$$

\begin{theorem}  Every positive rational is of the form $s_n$ for some $n$.\end{theorem}

\begin{proof}  Let $S:={\mathbb Q}^+$, $S_0:=\{1,2\}$, and define $F: S\rightarrow S$ by, for all $x\in S-S_0$,
$$F(x)=\max\left\{\frac x{1-x}, \frac{2x-2}{2-x}, x-2\right\}.$$
With, as before, $\Phi(a/b):=a+b$ when $a/b$ is in lowest terms, it is easy to verify that 
$\Phi(F(x))<\Phi(x)$
for all $x\in S-S_0$.  By the Tree Theorem, the positive rationals can be arranged in two rooted trees (here we identify $(a,b)$ with $\frac ba$):

\Tree[.(1,2) [.(1,4) ] [.(2,3) ]  [.(3,2) ] ].(1,2)
\Tree[.(1,1) [.(1,3) ] [.(3,4) ]  [.(2,1) ] ].(1,1)

By the definition of $(s_n)$, we can verify
$$\begin{cases}   s_{3n}&=s_n+2\\
s_{3n+1}&=\frac{2s_n+2}{s_n+2}\\
s_{3n+2}&=\frac{s_n}{1+s_n} \end{cases}\eqno{(15)}$$
and so, for all $n\ge 1$,
$$F(s_n)=s_{\lfloor n/3\rfloor}.$$
Hence every $s_n$ appears exactly once on the pair of trees and thus
$n\mapsto s_n$ is an enumeration of the positive rationals.  
  \end{proof}

\begin{corollary}
The iterates of $2+\frac 2x-4\left\{\frac1x\right\}$, starting at 2, span the entire set of positive rational numbers. \end{corollary}

\noindent{\bf Note.} This result  appeared \cite{monthly} and \cite{soln} as a problem and solution in the American Mathematical Monthly.  The sequence $(b_n)$ first appeared in \cite{N2} where most of the material in this section also appeared as well.

\subsection{Another example}
Here we have another array of circles:
$$C_{0,1}||C_{1,\sqrt 3}||C_{\sqrt 3, 2}||C_{2,\sqrt 3}||C_{\sqrt 3,1}||C_{1,0}$$
 \begin{figure}[htbp] 
   \centering
   \includegraphics[width=3.0in]{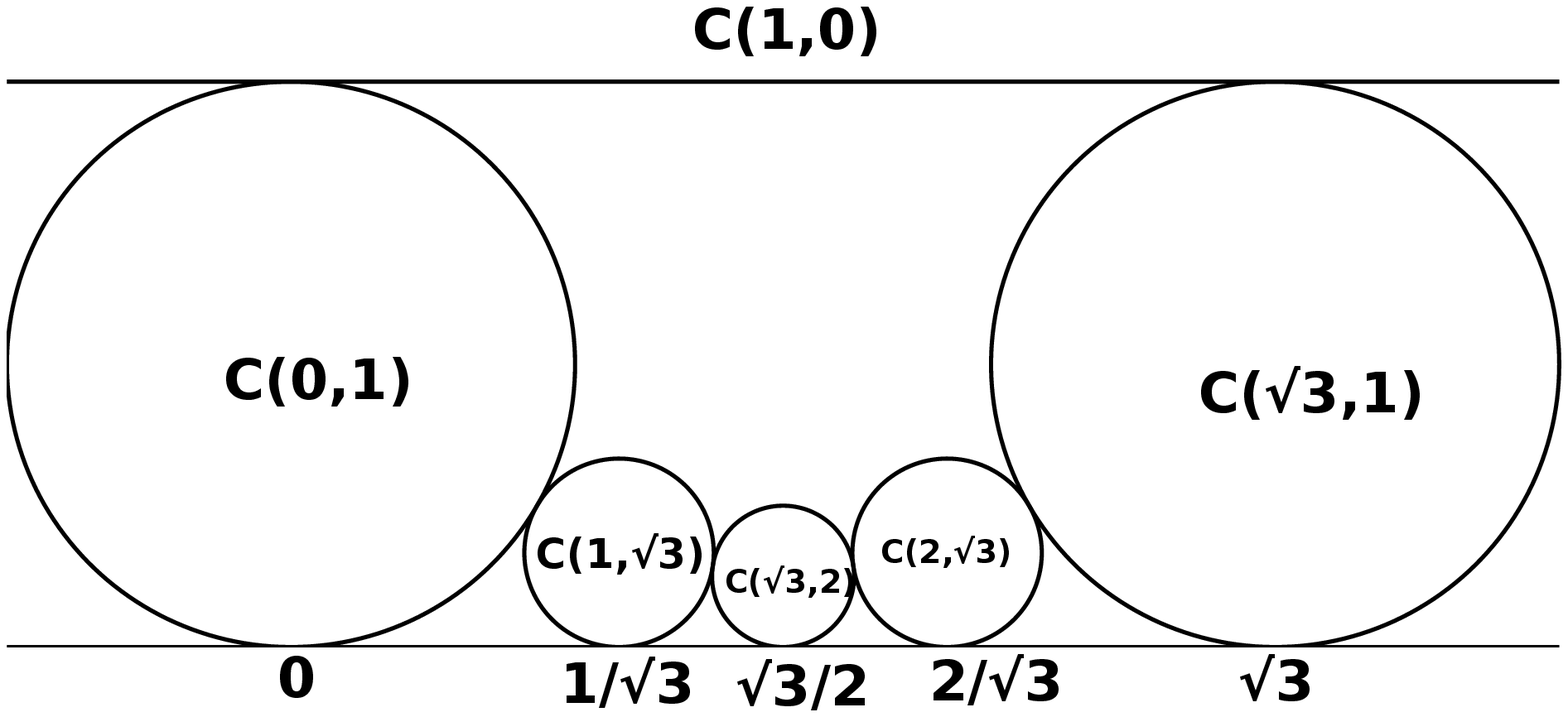}
  \end{figure}
  
\noindent giving rise to the recursive process
$$\begin{array}{cccccc} C_{a,b}&&&&&C_{c,d}\\
C_{a,b}&C_{a+c\sqrt 3, b+d\sqrt 3}&C_{a\sqrt 3+2c, b\sqrt 3+2d}&C_{2a+c\sqrt3, 2b+d\sqrt 3}&C_{a\sqrt 3+c, b\sqrt 3+d}&C_{b,d}\end{array}$$
which leads to 
 $$F(x)=\max\left\{\frac x{1-x}, \frac{3x-3}{3-2x}, \frac{2x-3}{2-x},\frac{3x-6}{3-x}, x-3\right\},$$ 
and thus to another enumeration of rationals.   

Here are some details;   our analogue of Stern's sequence here is
$$\begin{cases} c_{5n}&=c_n,\\ c_{5n+1}&=\sqrt 3\cdot c_n+c_{n+1}\\
 c_{5n+2}&=2c_n+\sqrt 3c_{n+1}\\
 c_{5n+3}&=\sqrt 3c_n+2c_{n+1}\\
c_{5n+4}&=c_n+\sqrt 3c_{n+1}\\
 \end{cases}\eqno{(16)}$$ 
 Let $$t_n:=\dfrac{\sqrt 3 c_{n+1}}{c_n}.$$
\begin{theorem} Every positive rational is of the form $t_n$ for some $n$.\end{theorem}
 \begin{proof} As above, we can use the Tree Theorem with $$F(x)=\max\left\{\frac x{1-x}, \frac{3x-3}{3-2x}, \frac{2x-3}{2-x},\frac{3x-6}{3-x}, x-3\right\},$$ 
 fixed points $S_0=\{1,\frac32,2,3\}$,  and  $\Phi$ as before.\end{proof}

As above, we can find a recursive formula:  $t_0=\infty$ and 
$$t_{n+1}=3\left(2\left\lfloor\dfrac1{t_n}\right\rfloor+1-\dfrac1{t_n}\right)$$
and we have the semi-recursive formula
$$t_n=3\left(2\nu_5(n)+1-\dfrac1{t_{n-1}}\right).$$

\subsection{A limit to generality}
The circular arrays are all examples of  ``necklaces":   the points of tangency are all on a circle.   In particular,  the iterates of the M\"obius transformation $$\begin{pmatrix}0&1\\-1&\alpha\end{pmatrix}(x):=\dfrac 1{\alpha-x}$$
are periodic for $\alpha=1, \sqrt 2, \phi \text{ (the golden ratio $\frac{\sqrt 5+1}2$)}$, and $\sqrt 3$ (we have necklaces of length 3,4,5, and 6 respectively).

\begin{figure}[htbp] 
   \centering
   \includegraphics[width=1.0in]{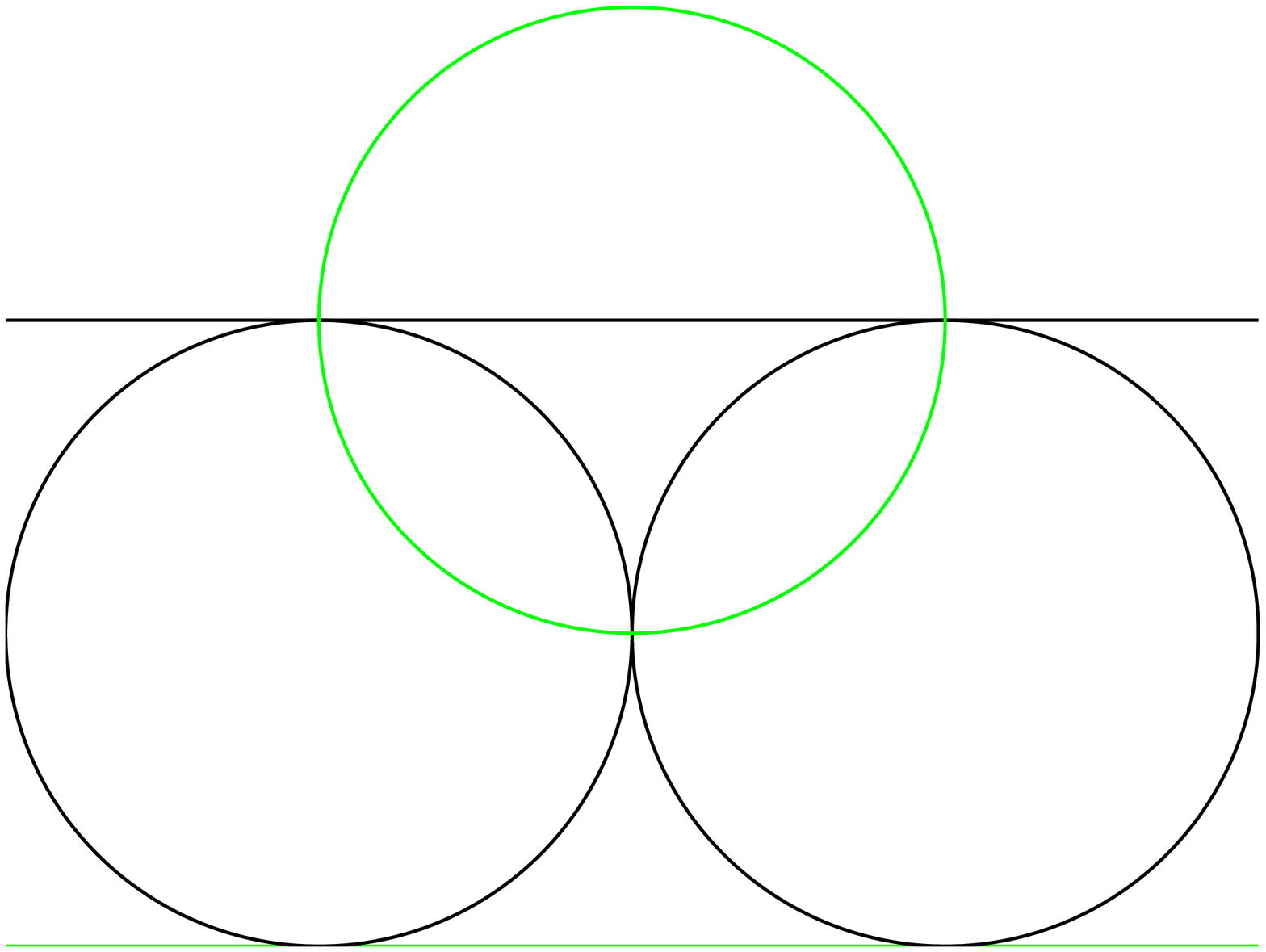}\hskip .5 in
   \includegraphics[width=1.0in]{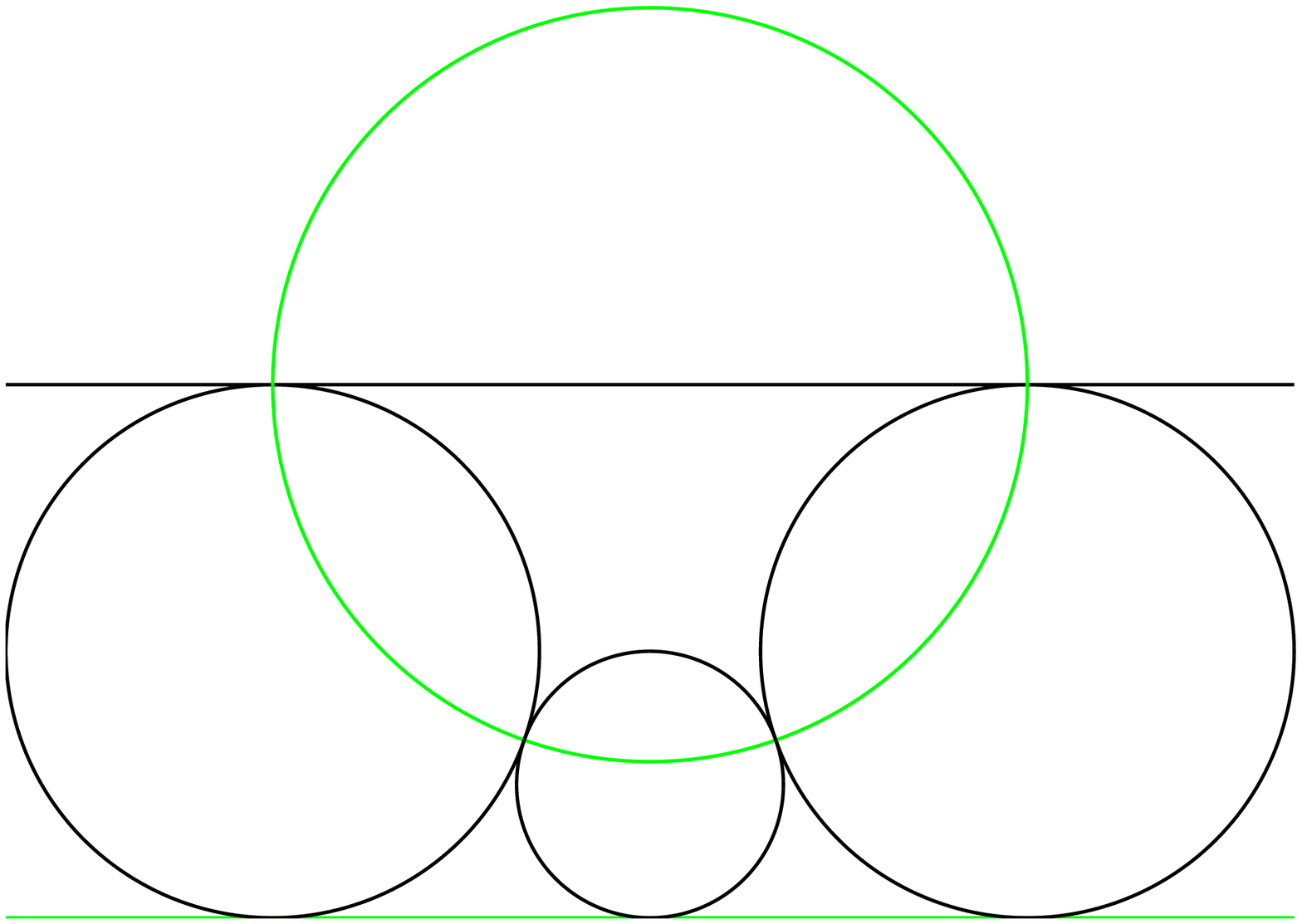}
  \end{figure} 
  
  \begin{figure}[htbp] 
   \centering
   \includegraphics[width=1.0in]{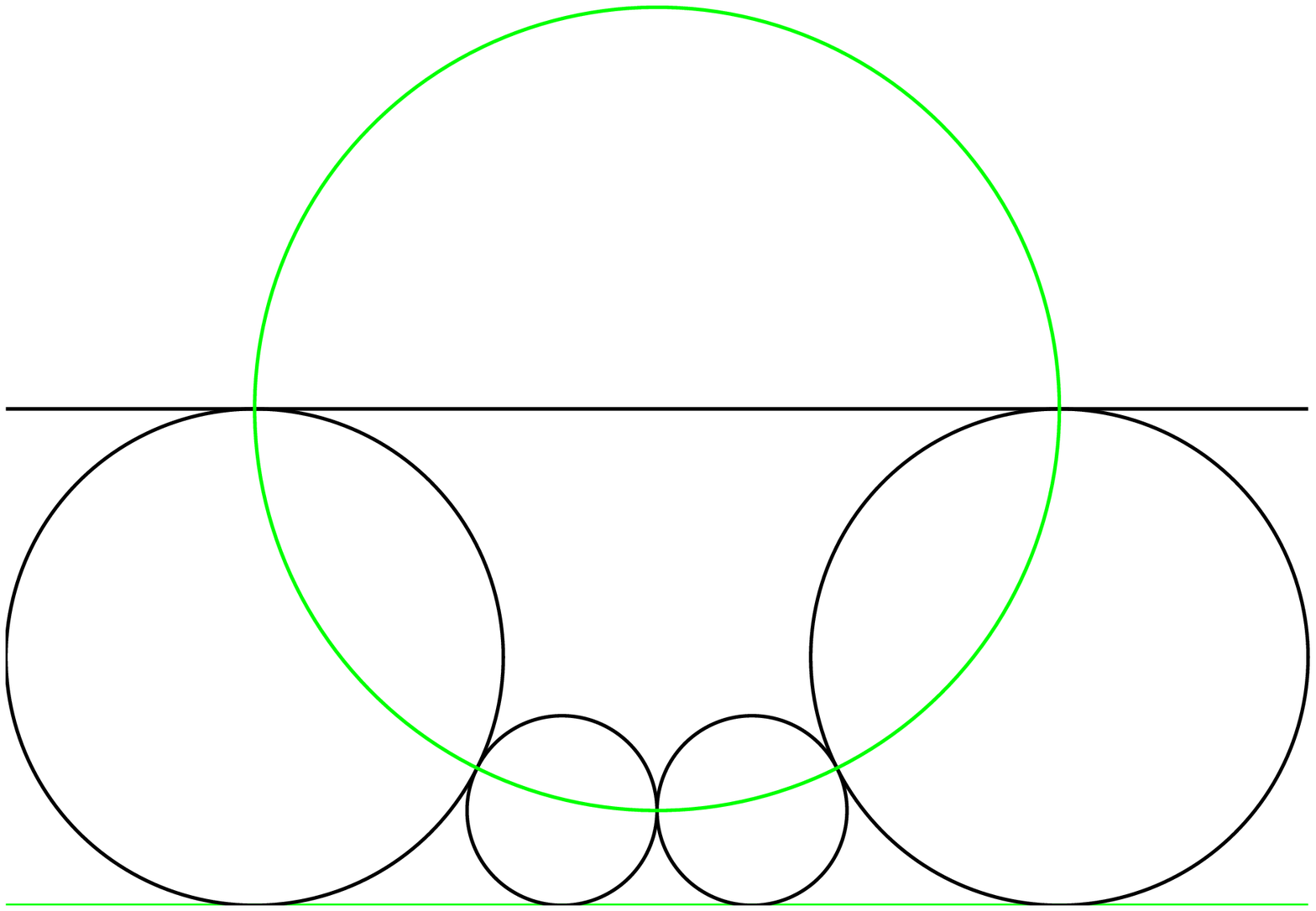}\hskip .5 in
   \includegraphics[width=1.0in]{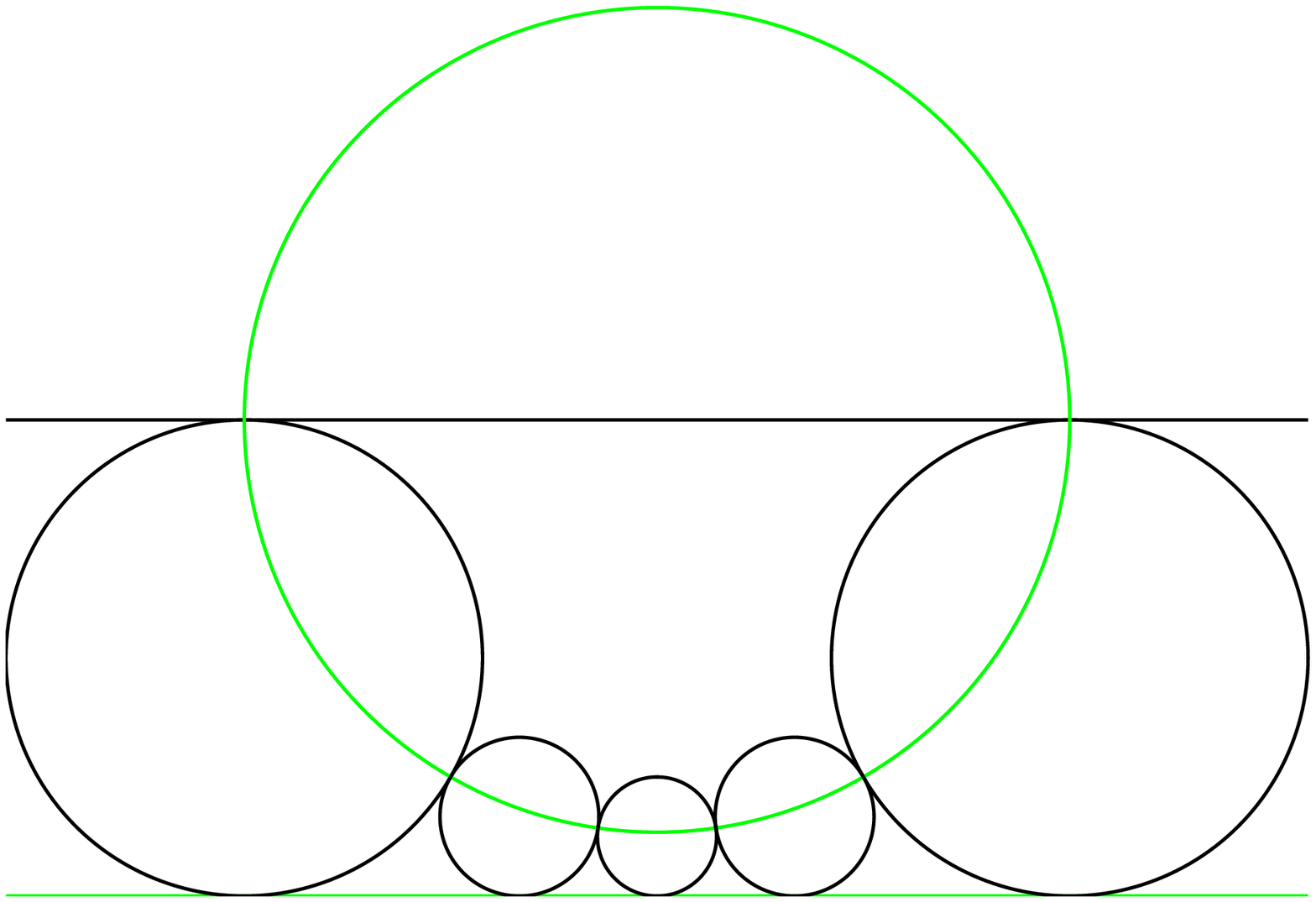}
  \end{figure}

For a real $\alpha$,  we can start a chain:
$$C(0,1)||C(1,\alpha)||C(\alpha, \alpha^2-1)||C(\alpha^2-1,\alpha^3-2\alpha)||\ldots$$
This is a sequence $C(p_{n-1}(\alpha), p_n(\alpha))$ where $(p_n)$ is a sequence of polynomials that satisfy a recurrence similar to those of Chebyshev polynomials and, in fact, $p_n(x)=U_n(x/2)$ where $U_n$ are Chebyshev polynomials of the second kind.  Furthermore, if $\alpha$ is the largest root of $p_N$ for some $N$, then, and only then, the chain terminates with $C(1,0)$.   Such necklaces give rise to an enumeration of the positive rationals via the recursion
$$x_{n+1}=\alpha^2\left(2\left\lfloor\dfrac1{x_n}\right\rfloor+1-\dfrac1{x_n}\right)$$
only when $\alpha^2$ is rational and thus $U_N$ must have largest irreducible factor of degree at most 2.
It is known \cite{Lehmer} that $U_n$ has largest irreducible factor of degree $\phi(2n+2)/2$ where $\phi(n)$ is the Euler totient function.   Starting with $n=1$, the sequence $\phi(2n+2)/2$ goes   $1,1,2,2,2,3,4,3,4,5,\ldots$;  the only quadratic or linear cases correspond to $n=1,2,3,4,5$.   The cases for $n=2,3,5$ are spoken for above, leaving the $n=4$ as the last quadratic case.  

The three enumerations of rationals we have found are the same as those Ponton \cite{P} found, independently, by using analogues of Calkin-Wilf trees.   Ponton further concludes, as did we, that these are the only three possibilities.  

\subsection{The $k=4$ case}
Based on the array
 \begin{figure}[htbp] 
   \centering
   \includegraphics[width=3 in]{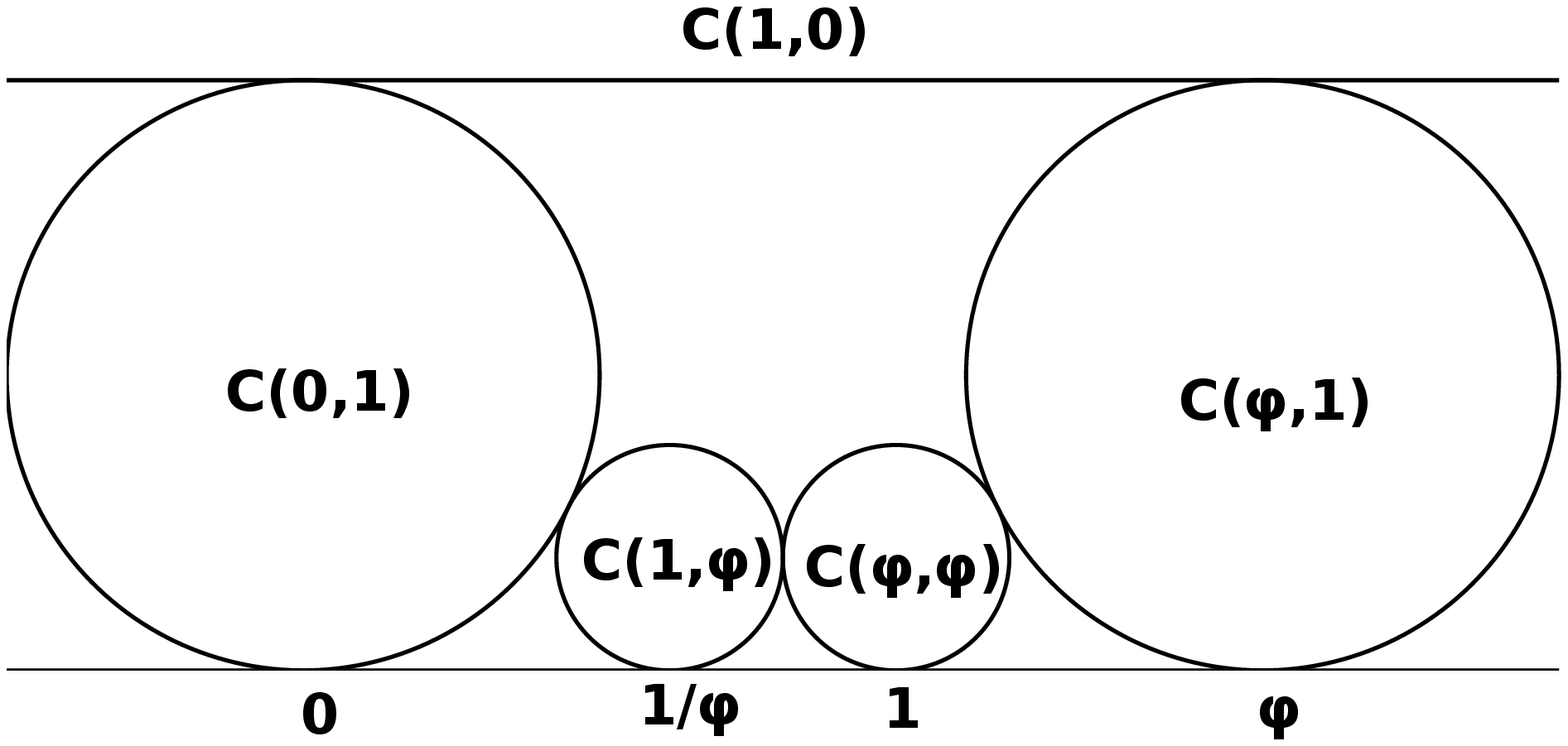}
    \end{figure}
    
\noindent we have the necklace
 $$C_{0,1}||C_{1,\phi}||C_{\phi,\phi}||C_{\phi,1}||C_{1,0}$$
 which indicates the recursive process
 $$\begin{array}{ccccc}
  C_{a,b}&&&&C_{c,d}\\
C_{a,b}&C_{a\phi+c, b\phi+d}&C_{a\phi+c\phi, b\phi+d\phi}&C_{a+c\phi, b+d\phi}&C_{c,d}\end{array}$$
 
\noindent and the analogue of Stern's sequence $d_0=0, d_1=1$ and, for $n\ge 0$

$$\begin{cases} d_{4n}&=d_n\\ d_{4n+1}&=\phi d_n+d_{n+1}\\
 d_{4n+2}&=\phi d_n+\phi d_{n+1}\\
 d_{4n+3}&=d_n+\phi d_{n+1}
 \end{cases}\eqno{(17)}$$ 
 The sequence $(d_n)$
begins

$1, \phi, \phi, 1, 2\phi, 1+2\phi, 2+\phi, \phi, 1+2\phi, 2+2\phi, 1+2\phi, \phi, 2+\phi, 1+2\phi, 2\phi, 1, 3\phi, 2+3\phi, 3+2\phi, 2\phi, 3+4\phi, 4+5\phi, 2+5\phi, 1+2\phi, 4+4\phi, 3+6\phi, 2+5\phi, 2+\phi, 1+4\phi, 2+4\phi, 3+2\phi, \phi, 2+3\phi, 3+4\phi, 2+4\phi, 1+2\phi, 4+5\phi, 4+7\phi, 3+6\phi, 2+2\phi, 3+6\phi, 4+7\phi, 4+5\phi, 1+2\phi, 2+4\phi, 3+4\phi, 2+3\phi, \phi, 3+2\phi, 2+4\phi, 1+4\phi, 2+\phi, 2+5\phi, 3+6\phi, 4+4\phi, 1+2\phi, 2+5\phi, 4+5\phi, 3+4\phi, 2\phi, 3+2\phi, 2+3\phi, 3\phi, 1, 4\phi, 3+4\phi, 4+3\phi, 3\phi, 5+6\phi, 6+8\phi, 3+8\phi, 2+3\phi, 6+7\phi, 5+10\phi, 4+8\phi, 3+2\phi, 2+7\phi, 4+7\phi, 5+4\phi, 2\phi, 5+6\phi, 6+9\phi, 4+9\phi, 3+4\phi, 8+12\phi, 9+16\phi, 8+13\phi, 4+5\phi, 7+14\phi, 10+16\phi, 9+12\phi, 2+5\phi, 6+9\phi,$

\vskip .1 in \noindent and, letting $u_n:=\dfrac{\phi d_{n+1}}{d_n},$
we have a sequence of elements in  $\Q^+(\phi)$ that begins:

\noindent $1+\phi,\phi,1,2+2\phi,1/2+\phi,3-\phi,2/5+\phi/5,1+2\phi,2,1/2+\phi/2,-1+\phi,2+\phi,3/5+4/5\phi,-2+2\phi,1/2,3+3\phi,2/3+\phi,-5+4\phi,{6/11}+2/11\phi,3/2+2\phi,-1/5+7/5\phi,{10/11}+3/11\phi,1/11+4/11\phi,4,3/4+3/4\phi,-1/3+\phi,{8/11}-\phi/11,7/5+6/5\phi,4/11+{10/11},2-\phi/2,3/11+\phi/11,2+3\phi,-1+2\phi,4/5+2/5\phi,\phi/2,3+\phi,{8/11}+{9/11},-3/5+6/5\phi,2/3,3/2+3/2\phi,1/3+\phi,8/5-\phi/5,3/11+2/11\phi,2\phi,1+\phi/2,1/5+3/5\phi,2-\phi,3+2\phi,{8/11}+{10/11},-1+3/2\phi,{7/11}+\phi/11,\ldots$
\vskip .1 in

By the definition of $(d_n)$ and $(u_n)$,

$$\begin{cases} u_{4n}&=\phi^2+u_n\\ 
u_{4n+1}&=\dfrac{\phi^3+\phi^2 u_n}{\phi^2+u_n} \\
 u_{4n+2}&=\dfrac{\phi+\phi u_n}{\phi+u_n}\\
 u_{4n+3}&=\dfrac {u_n}{1+u_n}
 \end{cases}\eqno{(18)}$$ 
 \noindent and so
 $$F(x):=\max\left\{ \dfrac x{1-x}, \dfrac{\phi x-\phi}{\phi-x}, \dfrac{\phi^2x-\phi^3}{\phi^2-x}, x-\phi^2\right\}$$
 satisfies $F(u_n)=u_{\lfloor n/4\rfloor}$ for all $n$.

Since $\Z[\phi]$ is a UFD, we may define, as above, $\Phi(a/b)=a+b$ where $a/b$ is in lowest terms over $\Z[\phi]$.  It is then easy to check that $$\Phi(F(x))<\Phi(x)\eqno{(19)}$$
for all $x\in\Q^+(\phi)$.    It is not clear, however,  if every element of $\Q^+(\phi)$ appears on this list since the range of $\Phi$ is not well-ordered -- the problem being that $\phi$ and all its integer powers are units in  $\Z[\phi]$.  We thus can only conjecture that
the sequence $(u_n)$ is an enumeration of $\Q(\phi)^+$.  
\vskip .1 in
\textbf{Note.}  A proof of this conjecture has resisted several attempts by the author.   Although computer evidence indicates it is true, it is of course possible for a counterexample to exist.   
\vskip .1 in

As above, we can find a recursive formula:
$$u_0=\infty \text{ and }u_{n+1}=\phi^2\left(2\left\lfloor\dfrac 1{u_n}\right\rfloor+1-\dfrac 1{u_n}\right).$$  
and we also have a semi-recursive formula $$u_n=\phi^2\left(2\nu_4(n)+1-\dfrac1{u_{n-1}}\right).$$

\section{Similarities and Differences}
In this section, we summarize some of the similarities and differences  between our various sequences.  

\subsection{Trees}

The four sequences we have studied each could be arranged on the vertices of trees defined, via the Tree Theorem, by a certain function.   
Here our functions are:
\begin{itemize}
\item $r_n: f(x)=\max\left\{\frac x{1-x}, x-1\right\},$
\item $s_n: f(x)= \max\left\{\frac x{1-x}, \frac{2x-2}{2-x}, x-2\right\},$
\item $t_n:  f(x)=\max\left\{\frac x{1-x}, \frac{3x-3}{3-2x}, \frac{2x-3}{2-x},\frac{3x-6}{3-x}, x-3\right\}.$
\item $u_n: f(x)=\max\left\{ \dfrac x{1-x}, \dfrac{\phi x-\phi}{\phi-x}, \dfrac{\phi^2x-\phi^3}{\phi^2-x}, x-\phi^2\right\}.$
\end{itemize}

A method of finding them is as follows.  
Given $\alpha\in\{1,\sqrt 2, \phi, \sqrt 3\}$,  let $M$ be the M\"obius transformation defined by $\begin{pmatrix} 1&\alpha\\0&-1\end{pmatrix}$, and let
$a_n=\alpha\cdot M_n(0)$  where $0=a_0<a_1<\ldots<a_N=\infty$.   The corresponding $F$ is given by
$$F(x)=\max\left\{ a_j\cdot\dfrac{x-a_{j-1}}{a_j-x}:  j=1\ldots N\right\}.$$

\subsection{Analogues of Stern's sequence}

For  $\alpha=1,\sqrt 2, \sqrt 3, \phi$, our sequences all have similar 3-term recurrences:

\begin{itemize}
\item $a_{n+1}=a_n+a_{n-1}-2(a_{n-1}\text{ mod } a_n),$
\item $b_{n+1}=\sqrt 2 b_n+b_{n-1}-2(b_{n-1}\text{ mod } (\sqrt 2 b_n)),$
\item $c_{n+1}=\sqrt 3 c_n+c_{n-1}-2(c_{n-1}\text{ mod } (\sqrt 3 c_n)),$
\item $d_{n+1}=\phi d_n+d_{n-1}-2(d_{n-1}\text{ mod } (\phi d_n)).$
\end{itemize}

From these, our sequences of rationals are as follows  (the last, of course, in ${\mathbb Q}(\phi))$): 

\begin{itemize}

\item $r_n=\dfrac{a_{n+1}}{a_n},$

\item $s_n=\dfrac{\sqrt 2 b_{n+1}}{b_n},$

\item $t_n:=\dfrac{\sqrt 3 c_{n+1}}{c_n},$

\item $u_n:=\dfrac{\phi d_{n+1}}{d_n}.$

\end{itemize}

\subsection{Semi-recursive formulas}
The sequences in the text all have semi-recursive formulas involving valuation functions $\nu_N(n)$.

Starting with $\infty$, 
\begin{itemize}
\item $r_n=1\left(2\nu_2(n)+1-\dfrac1{r_{n-1}}\right),$
\item $s_n=2\left(2\nu_3(n)+1-\dfrac1{s_{n-1}}\right),$
\item $t_n=3\left(2\nu_5(n)+1-\dfrac1{t_{n-1}}\right),$
\item $u_n=\phi^2\left(2\nu_4(n)+1-\dfrac1{u_{n-1}}\right).$
\end{itemize}

In fact, they are all of the form
$$x_{n}=4\cos^2\left(\frac{\pi}{N+1}\right)\cdot \left(2\nu_N(n)+1-\dfrac1{x_{n-1}}\right)$$
where $N=2,3,5$ and 4 respectively.     The corresponding sequences for $N>5$ surely share many of the properties of the sequences we have studied.  

\subsection{Recurrence formulas}
Let $$f(x):=1+\dfrac1x-2\left\{\dfrac 1x\right\}=2\left\lfloor\dfrac 1x\right\rfloor+1-\dfrac1x.$$
The 3-term recurrences above give:
\begin{itemize}
\item $r_{n+1}=f(r_n),$
\item $s_{n+1}=2f(s_n),$
\item $t_{n+1}=3f(t_n),$
\item $u_{n+1}=\phi^2 f(u_n).$
\end{itemize}
That is, all are of the form $$x_{n+1}=\alpha^2 f(x_n)$$
for $\alpha= 1, \sqrt 2,\sqrt 3, \phi$ respectively.

\subsection{Greedy algorithms}

The sequences above all can be obtained by greedy algorithms:  
\begin{itemize}

\item $r_n:  [l_1,\ldots,l_k]\mapsto[l_1,\ldots,l_k,n-1/l_k],$

\item $s_n:  [l_1,\ldots,l_k]\mapsto[l_1,\ldots,l_k,2(n-1/l_k)],$

\item $t_n:  [l_1,\ldots,l_k]\mapsto[l_1,\ldots,l_k,3(n-1/l_k)],$

\item $u_n:  [l_1,\ldots,l_k]\mapsto[l_1,\ldots,l_k,\phi^2(n-1/l_k)],$

\end{itemize}

where, in each case, $n$ is the least positive integer making the elements in the list distinct.

\subsection{Degrees}
The \emph{degree} of a sequence $(x_n)$ is a real number $\alpha$ such that for all $\epsilon>0$, 
$$\limsup_{n\rightarrow\infty}\dfrac{x_n}{n^{\alpha+\epsilon}}=0\qquad \text{ and } \qquad\limsup_{n\rightarrow\infty}\dfrac{x_n}{n^{\alpha-\epsilon}}=\infty.$$
The term comes from the fact that, for a polynomial $p$ of degree $N$, the sequence $(p(n))$ has degree $N$.  For a general sequence, this number need not exist.    However, if it does exist, then it is unique.

\begin{itemize} 
\item  $(a_n)$ has degree $\log_2(\phi)=0.694241914\ldots$.   This result follows from a stronger result by Coons and Tyler \cite{CT}.    It is also fairly easy to prove from basic properties of Stern's sequence \cite{N1} .

\item $(b_n)$ has degree $\log_3(1+\sqrt 2)=0.802260812\ldots$.  This result follows from a stronger result by Coons \cite{Coons}.    It is also fairly easy to prove from basic properties of $(b_n)$ that can be found in \cite{N2}.

\item $(c_n)$ has degree $\log_5(2+\sqrt 3)=0.818271949\ldots$.   This has been proved by the author using methods similar to those used in the previous cases (unpublished).

\item$(d_n)$ seems to have degree $\log_4\left(\dfrac{\phi^2+\sqrt{4+\phi^4}}2\right)=.7818951685\ldots$.  This result seems to follow by the same methods as above.   The result is surprising since one might expect the value to between those for $(b_n)$ and $(c_n)$ (i.e., between 0.802 and 0.819).   

\end{itemize}

\subsection{Generating functions}
The generating functions for these sequences each have product formulas easily deduced from the following expressions.   
By observation, the polynomials involved have each of their zeros equal to some root of unity (we call them ``primary roots").   
\begin{itemize}
\item$A(x)=\sum a_{n+1}x^n=\left(1+x+x^2\right)A(x^2),$

primary roots $e^{i\pi n/6}$, $n=4,8$.\vskip .1in

\item $B(x)=\sum b_{n+1}x^n=\left(1+\sqrt 2x+x^2\sqrt 2 x^3+x^4\right)B(x^3),$
primary roots $e^{i\pi n/12}$, $n=5,11,13,19$.\vskip .1in

\item $C(x)=\sum c_{n+1}x^n=\left(1+\sqrt 3x+2x^2+\sqrt 3x^3+x^4+\sqrt 3x^5+2 x^6+\sqrt3x^7+x^8\right)C(x^5),$

primary roots $e^{i\pi n/30}$, $n=7,17,19,29,31,41,43,53$.
 
\item $D(x)=\sum d_{n+1}x^n=\left(1+\phi x+ \phi x^2+ x^3+\phi x^4+\phi x^5+x^6 \right)D(x^4)$

primary roots $e^{i\pi n/20}$,  $n=6,14,16,24,26,34$.

\end{itemize}
We see that the sets of primary roots are, for $n=2,3,5,4$ respectively,
 $$\left\{e^{i\pi r}:   r=\tfrac {2j}n\pm\tfrac1{n+1}, j=1,\dots, n-1\right\}.$$

\subsection{Closed formulas}

Letting $\langle x_1,x_2,\ldots,x_n\rangle_k$ be 1 or 0 according as the $k$-ary expansions of $x_1,\ldots, x_n$ share no non-zero digits, recall by equation (10)

$$a_{n+1}=\sum_{a+2b=n} \langle a,b\rangle_2.$$
This is equivalent to the fact that diagonal sums across Pascal's triangle mod 2 give $a_{n+1}$ (where the unmodded sums give $F_{n+1}$).\vskip .1 in

By \cite{N2}, Th. 10
$$b_{n+1}=\sum_{a+2b+3c+4d=n} \langle a,b,c,d\rangle_3\sqrt 2^{a+c}$$ 
and it is not hard to see that 
$$c_{n+1}=\sum_{a+2b+3c+4d+5e+6f+7g+8h=n} \langle a,b,c,d,e,f,g,h\rangle_5\sqrt 3^{a+c+e+g}2^{b+f}.$$

Recall Binet's formula:
$$F_{n+1}=\sum_{k=0}^n \phi^{s(k)}\overline\phi^{s(n-k)}$$
where $\phi,\overline\phi$ are zeros of $x^2-x-1$ and $s(n)=n$.

The sequences $(a_n)$ and $(b_n)$ have Binet type formulas: by equation (11),
$$a_{n+1}=\sum_{k=0}^n \sigma^{s(k)}\overline\sigma^{s(n-k)}$$
where $\sigma,\overline\sigma$ are zeros of $x^2+x+1$ and $s(n)$ is the number of ones in the binary expansion of $n$. \vskip .1 in

By \cite{N2}, Th. 12
$$b_{n+1}=\sum_{k=0}^n \tau^{s(k)}\overline\tau^{s(n-k)}$$
where $\tau,\overline\tau$ are zeros of $x^2-\sqrt 2x-1$ and $s(n)$ is the number of ones in the ternary expansion of $n$. 

The proofs are from the formulas for the generating functions;   it is not clear if $(c_n)$ and $(d_n)$ have similar formulas.

\subsection{Singular functions}

As noted in \cite{N1},
$f\left(\frac n{2^k}\right):=\dfrac{a_n}{a_{2^k+n}}$
is Conway's box function with $f^{-1}(x)= ?(x)$, Minkowski's question mark function
defined by 
$$f^{-1}([0,c_1,c_2,\ldots])=\sum\frac{(-1)^{k+1}}{2^{c_1+\ldots+c_k-1}}.$$
\vskip .1 in

As noted in \cite{N2},
$g\left(\frac n{3^k}\right):=\dfrac{b_n}{b_{3^k+n}}$
has a singular inverse function $g^{-1}(x)$ satisfying
$$g^{-1}([0,c_1\sqrt 2,c_2\sqrt 2,\ldots])=\sum\frac{(-1)^{k+1}}{3^{c_1+\ldots+c_k-1}}.$$
\vskip .1 in

Based on these results, it seems likely that 
$h\left(\frac n{5^k}\right):=\dfrac{c_n}{c_{5^k+n}}$
has a singular inverse function $h^{-1}(x)$ satisfying
$$h^{-1}([0,c_1\sqrt 3,c_2\sqrt 3,\ldots])=\sum\frac{(-1)^{k+1}}{5^{c_1+\ldots+c_k-1}}$$
\vskip .1 in
\noindent and that
$l\left(\frac n{4^k}\right):=\dfrac{d_n}{d_{4^k+n}}$
has a singular inverse function $l^{-1}(x)$ satisfying
$$l^{-1}([0,d_1\phi,d_2\phi,\ldots])=\sum\frac{(-1)^{k+1}}{4^{d_1+\ldots+d_k-1}}.$$

\subsection{Minus continued Fractions}

``Minus continued fractions" are of the form $$(a_0,a_1,a_2,\ldots):=a_0-\frac 1 {a_1-\frac 1{ a_2-\frac 1{a_3-\ldots}}}$$ where $a_i\in{\mathbb Z}$.     
\vskip .1 in

Every positive rational is a minus continued fraction of the form $$u_n-\frac 1 {u_{n-1}-\frac 1{ u_{n-2}-\frac 1{u_{n-3}-\ldots}}}=(u_n, u_{n-1}, u_{n-2}, u_{n-3}, \ldots, u_1)$$  where $u_n=2\nu_2(n)+1$.  

Every positive rational is a minus continued fraction of the form $$2u_n-\frac 2 {2u_{n-1}-\frac 2{ 2u_{n-2}-\frac 2{2u_{n-3}-\ldots}}}=(2u_n, u_{n-1}, 2u_{n-2}, u_{n-3}, \ldots, cu_1)$$  where $u_n=2\nu_3(n)+1$ and $c$ is 1 or 2, depending on the parity of $n$.  

Every positive rational is a minus continued fraction of the form $$3u_n-\frac 3 {3u_{n-1}-\frac 3{ 3u_{n-2}-\frac 3{3u_{n-3}-\ldots}}}=(3u_n, u_{n-1}, 3u_{n-2}, u_{n-3}, \ldots, cu_1 )$$  where $u_n=2\nu_5(n)+1$ and $c$ is 1 or 3, depending on the parity of $n$.

As with equation (8),  we can also express our rational sequences in terms of minus continued fractions defined by lists.   Let $\ell_0$ be the empty list $[\hskip .05 in]$ and, for a given $k\in\{2,3,4,5,\ldots\}$, let 
$$\begin{cases}
&\ell_{kn}:=1+\ell_n\\
&\ell_{kn+1}:=1*\ell_{kn}\\
&\ell_{kn+2}:=1*\ell_{kn+1}\\
&\ldots\ldots\ldots\ldots\ldots\ldots\ldots\\
&\ell_{kn+k-1}=1*\ell_{kn+k-2}\\
\end{cases}\eqno{(20)}$$

For our given $k$ and its corresponding lists $\ell_n=[\ell_{n,0}, \ell_{n,1}, \ell_{n,2},\ldots]$, let $\alpha=2\cdot\cos\left(\frac{\pi}{k+1}\right)$ and define a sequence $(x_n)$ in terms of minus continued fractions
$$x_n:=\alpha\cdot(\alpha \ell_{n,0},\alpha \ell_{n,1},\alpha \ell_{n,2},\ldots).$$
These sequences agree with $(a_n), (b_n), (c_n),(d_n)$ when $k=2,3,5,4$ respectively.

\end{document}